\documentclass[10pt]{article}
\usepackage{amsmath, amssymb, amsfonts, amsthm, bm, mathtools, dsfont}
\usepackage{tikz}
\usepackage{mathrsfs}

\usepackage{amssymb}
\usepackage{latexsym }
\usepackage{graphicx }
\usepackage{amsmath}
\usepackage{amsfonts}
\usepackage{bbm}
\usepackage{accents}
\usepackage{psfrag}

\usepackage{pgf}
\usepackage{tikz}
\usepackage{pgfplots}
\usetikzlibrary{arrows,automata,positioning}
\usepackage{amsmath}
\usetikzlibrary{calc}
\usepackage{amsthm}
\usepackage{thm-restate}

\usepackage{algorithm}
\usepackage{algorithmicx}
\usepackage{algpseudocode}

\let\originalleft\left
\let\originalright\right
\renewcommand{\left}{\mathopen{}\mathclose\bgroup\originalleft}
\renewcommand{\right}{\aftergroup\egroup\originalright}

\newcommand{\abs}[1]{\left\vert #1 \right\vert}

\newcommand{\specrad}[1]{\sigma ( #1 )}

\newcommand{\eps}{\epsilon}

\theoremstyle{plain}

\newtheorem{theorem}{Theorem}
\newtheorem{condition}{Condition}
\newtheorem{definition}{Definition}

\newtheorem{claim}{Claim}[section]

\newtheorem{lemma}[claim]{Lemma}

\title{Traffic Equations for Fluid Networks with Overflows}
\author{S.\ Fleuren
	\thanks{
		Swift Mobility B.V., Sioux Lime, and Department of Mechanical Engineering, Eindhoven University of Technology, PO Box 513, 5600MB, Eindhoven, The Netherlands.
	},
	H.\ M.\ Jansen
	\thanks{
         Department of Applied Mathematics, Delft University of Technology, Delft, the Netherlands. 
		 Corresponding author: \texttt{h.m.jansen@tudelft.nl}
	},
	E.\ Lefeber
	\thanks{
		Department of Mechanical Engineering, Eindhoven University of Technology, PO Box 513, 5600MB, Eindhoven, The Netherlands.
	},
	Y.\ Nazarathy
	\thanks{
		School of Mathematics and Physics, The University of Queensland, St Lucia, Queensland 4072, Australia
	}
}

\begin{document}
\maketitle

\begin{abstract}
\noindent Many real life queueing networks have finite buffers with overflows. To understand the behavior of such networks, we consider traffic equations that generalize the traffic equations of classic open queueing networks where some nodes are potentially overloaded. We present a novel, efficient algorithm for solving the equations for overflow networks together with a sufficient condition for existence and uniqueness of solutions. Our analysis also sharpens results of traffic equations for classic open queueing networks.\\ 
\\
\textbf{\footnotesize Keywords: } Fluid network $\bullet$ traffic equation $\bullet$ overflow $\bullet$ queueing system
\end{abstract}

\section{Introduction}
Analysis of queueing network models has taken a central role in the study of operations research dealing with resource-constrained dynamic systems. 
In a queueing network, {\em jobs} traverse through service stations at nodes, they queue up at occupied nodes, and they receive service before moving onwards. Many queueing network models are stochastic in nature, incorporating exogenous random arrival processes of jobs, random service durations and possibly random routing. The classical elementary model is the {\em open Jackson network}, in which jobs arrive to nodes according to Poisson processes, are served in each node by single-server nodes based on i.i.d.\ exponential processing time durations, and are routed probabilistically to other nodes, where the routing is not affected by the history of the network or any other aspect.

Jackson \cite{Jackson407} discovered a very elegant product-form solution for the stationary distribution of the number of jobs in each node. This is a result for the case where there is enough capacity in each server to serve the incoming jobs. Many subsequent product-form, stationary-distribution results followed for variants of this model, which was subsequently named the {\em Jackson network}. See for example \cite{bookChenYao2001} or Chapter~2 of \cite{bramsonBook2008} for an overview. Years after the discovery of Jackson's result, networks in which the service and routing patterns of jobs are indistinguishable (as in Jackson networks) have been termed {\em single-class}. This differentiates such networks from {\em multi-class} networks, in which individual nodes may serve more than one job class, each possibly with its own service requirements and routing pattern. Networks are termed {\em open} if jobs may arrive from the outside world or depart there. In this paper we deal with open single-class networks.

While the idea of finding product-form stationary distributions has had phenomenal success for single-class networks, there are many variations of stochastic queueing network models that are not explicitly solvable.  The standard example is the so called {\em generalized Jackson network}, in which arrival streams of jobs are general renewal processes and processing times follow arbitrary probability distributions.  In such cases, various approximation approaches have been proposed, both heuristic and having some theoretical basis. One such approximation is the so-called {\em fluid approximation}, analyzed exhaustively in \cite{chen1991dfn}.  In this case, job flows are replaced by idealized deterministic fluid flows, while discrete counts of the number of jobs in each buffer are replaced by continuous fluid tanks. This transforms the discrete stochastic queueing system into a continuous deterministic dynamical system with trajectories that are easily described as piece-wise affine, having break-points at time instances at which buffers change mode from being empty to not or vice versa.

In the case of a generalized Jackson network, the fluid model indicates which nodes are underloaded, which are overloaded and which are critical. This is the essence of the {\em bottleneck analysis} in \cite{chen1991dfn}. Such analysis is also used to find the flow rates between nodes and can be used as an aid for determining stability of the system, as in \cite{Dai108}. While such a macroscopic view of discrete-job systems is often considered too-coarse, it may also be employed as a modelling tool from the go. This is the case when studying stochastic fluid buffers, see for example \cite{zwart2000fluid} and references there-in, as well as when considering time-varying fluid systems, as in the classic work surveyed in \cite{newell1971aqt}.
%
%

A central component of bottleneck analysis involves {\em traffic equations}. For this consider the following three vector-valued equations
%
\begin{align}
\label{eq:basicJackson}
\lambda &= \alpha + \lambda P, \\
\label{eq:goodmanMasseyEq}
\lambda &= \alpha + \min(\lambda,\mu) P, \\
\label{eq:ourTrafficEquation}
\lambda &= \alpha + \min(\lambda,\mu) P + \max(\lambda-\mu,0) Q.
\end{align}
Here $\alpha$ and $\mu$ are $n$-dimensional nonnegative row vectors, while $P$ and $Q$ are $n \times n$ sub-stochastic routing matrices. 
The minimum and maximum operations are taken element-wise, and $0$ is a row vector of $0$'s. The {\em unknown} in these equations is the row vector $\lambda$, where $\lambda_i$ indicates the total arrival rate to node~$i$.

The first two systems of equations, \eqref{eq:basicJackson} and \eqref{eq:goodmanMasseyEq}, have been previously studied extensively. The system \eqref{eq:basicJackson} is well-known for Jackson networks (as well as generalized Jackson networks), in which $P$ is the routing matrix and $\alpha$ is the exogenous arrival rate of jobs. Usage of its solution assumes that all nodes are able to sustain their arrivals and hence arrivals at a rate of $\lambda_i$ imply departures at a rate of $\lambda_i$. Indeed this is the case if $\lambda_i<\mu_i$ for all nodes $i$, with $\mu_i$ denoting the service capacity of node $i$.  It is clear that the meaning of equation \eqref{eq:basicJackson} breaks down once nodes become overloaded.
%
The more general system \eqref{eq:goodmanMasseyEq} was advanced by Goodman and Massey in \cite{goodman1984nej} and is also attributed to \cite{schweitzer1982bottleneck}. In their work, besides showing under a specific condition that \eqref{eq:goodmanMasseyEq} has an efficiently found unique nonnegative solution, Goodman and Massey also generalized Jackson's product-form result to allow for the case in which some nodes have $\lambda_i \geq \mu_i$. They found a product-form solution (in the case of memoryless assumptions) on the nodes in which $\lambda_i < \mu_i$, and showed that other nodes diverge to infinity with probability $1$. Also see \cite{sommer2017analysis}, where a generalization of \eqref{eq:goodmanMasseyEq} is introduced to study networks with unreliable nodes. Chen and Mandelbaum \cite{chen1991dfn} relate \eqref{eq:goodmanMasseyEq} to a linear complementarity problem (LCP). See \cite{cottlelinear} or \cite{murty1988linear} for LCP background.

Our main contribution is the introduction and analysis of the third system of equations \eqref{eq:ourTrafficEquation}. Our motivation for this natural extension is that \eqref{eq:ourTrafficEquation} allows us to incorporate {\em overflow phenomena} in models and to carry out a bottleneck analysis in the presence of finite buffers and overflows. In formulating these equations we assume nodes may have a finite buffer. Further, jobs (or material) arriving to nodes with {\em full} buffers overflow to other nodes with proportions $q_{ij}$. In this case, spelling out \eqref{eq:ourTrafficEquation} for an individual node $i$, we have
\begin{align*}
	\lambda_i = \alpha_i + \sum_{j=1}^{n}  \min(\lambda_j,\mu_j) p_{ji} + \sum_{j=1}^{n} \max(\lambda_j-\mu_j,0) q_{ji},
	\quad
	i=1,\dotsc,n.
\end{align*}
Here, $\lambda_{i}$ denotes the total arrival rate to node $i$ and is determined by the rate $\alpha_{i}$ of exogenous arrivals to node $i$ as well as the arrival rate to other nodes. More precisely, if node $j$ (which has capacity $\mu_{j}$) faces a total arrival rate of $\lambda_{j}$, then it processes the arrivals at rate $\lambda_{j} \wedge \mu_{j}$ and sends a fraction $p_{ji}$ to node $i$. If node $j$ faces a total arrival rate $\lambda_{j}$ that exceeds its capacity $\mu_{j}$, then overflow with rate $\lambda_{j} - \mu_{j}$ is created and a fraction $q_{ji}$ of this overflow is sent to node $i$. This gives rise to the overflow traffic equations \eqref{eq:ourTrafficEquation} for the total arrival rates.
Thus, the total arrival rate $\lambda_{i}$ to node $i$ equals the sum of the exogenous arrivals (first term on the right), the arrivals due to service (second term on the right) and the overflows resulting from finite buffer nodes with $\lambda_j>\mu_j$ (third term on the right). This behaviour is significantly different than the infinite-buffer case described by \eqref{eq:goodmanMasseyEq}.
Note that for infinite-buffer nodes, the corresponding row of $Q$ is taken as $0$ and the node fills up without bound.
%
A schematic illustration of the type of network we consider is in Figure~\ref{fig1}.

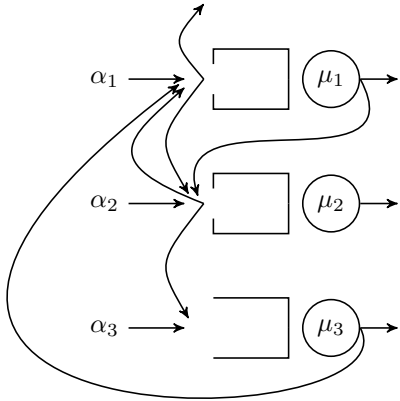
\begin{figure}[h]
\begin{tikzpicture}[>=stealth',auto,
                    semithick, every node/.style={transform shape}]
  \tikzstyle{every state}=[draw=black,text=black]
\node		    (1)  {};
\node		    (2) [below =1.4cm of  1]  {};
\node		    (3) [below =1.4cm of  2]  {};

\node (1s) [right=0.05cm of 1, shape = circle,draw]{$\mu_1$};
\node (2s) [right=0.05cm of 2, shape = circle,draw]{$\mu_2$ };
\node (3s) [right=0.05cm of 3, shape = circle,draw]{$\mu_3$};

\node (1buf) [left = 1cm of 1]{} ;
\node (2buf) [left = 1cm of 2]{} ;
\node (3buf) [left = 1cm of 3]{} ;

\node (1out)[right=0.5cm of 1s]{};
\node (2out)[right=0.5cm of 2s]{};
\node (3out)[right=0.5cm of 3s]{};

\node[left=2cm of 1] (1a) {$\alpha_1$};
\node[left=2cm of 2] (2a) {$\alpha_2$};
\node[left=2cm of 3] (3a) {$\alpha_3$};

\draw[->] (1a) -- (1buf);
\draw[->] (2a) -- (2buf);
\draw[->] (3a) -- (3buf);

\draw[->] (1s) -- (1out);
\draw[->] (2s) -- (2out);
\draw[->] (3s) -- (3out);

\draw (1.center) -- ($(1.center)+(0cm,0.4cm)$) -- ($(1.center)+(-1cm,0.4cm)$) -- ($(1.center)+(-1cm,0.2cm)$);
\draw (1.center) -- ($(1.center)+(0cm,-0.4cm)$) -- ($(1.center)+(-1cm,-0.4cm)$) -- ($(1.center)+(-1cm,-0.2cm)$);

	\draw (2.center) -- ($(2.center)+(0cm,0.4cm)$) -- ($(2.center)+(-1cm,0.4cm)$)-- ($(2.center)+(-1cm,0.2cm)$);
	\draw (2.center) -- ($(2.center)+(0cm,-0.4cm)$) -- ($(2.center)+(-1cm,-0.4cm)$)-- ($(2.center)+(-1cm,-0.2cm)$);

	\draw (3.center) -- ($(3.center)+(0cm,0.4cm)$) -- ($(3.center)+(-1cm,0.4cm)$);
	\draw (3.center) -- ($(3.center)+(0cm,-0.4cm)$) -- ($(3.center)+(-1cm,-0.4cm)$);
	
\draw [->] (1s.east) .. controls ($(1s.east) + (0.8cm,-1.5cm)$) and ($(2buf)+(-0.4cm,1.5cm)$) .. ($(2buf)+(0.05cm,0.1cm)$);

\draw [->] (2buf.east) .. controls ($(2buf) + (-0.5cm,-0.8cm)$) .. (3buf);

\draw [->] (1buf.east) .. controls ($(1buf) + (-0.5cm,-0.8cm)$) .. (2buf);

\draw [->] (2buf.east) .. controls ($(1buf) + (-1.3cm,-1.1cm)$) and ($(1buf) + (-0.9cm,-0.8cm)$) .. (1buf);

\draw [->] (1buf.east) .. controls ($(1buf) + (-0.3cm,0.5cm)$) .. ($(1buf.east)+(0,1cm)$);

\draw[->] (3s.east) .. controls ($(3s.east) + (0.8cm,-1.5cm)$) and  (-8,-5)  .. ($(1buf)+(-0.2,-0.07cm)$);
\end{tikzpicture}
\caption{\it An illustration of an overflow fluid network with nodes $i=1,2,3$. Fluid arrives exogenously at rates $\alpha_i$ and is served at rates $\mu_i$. Fluid is drained out of node $i$ and is routed to node $j$ with proportion $p_{ij}$.  Fluid reaching a node with a full buffer $i$ overflows to other nodes with proportion $q_{ij}$. The remaining proportions of routing ($p$) or overflow ($q$) leave the system. Nodes with infinite buffers, such as node $i=3$ in this example, have $\sum_{j} q_{ij} = 0$. \label{fig1}}
\end{figure}

Our motivation for considering overflows is that the infinite-buffer assumption of standard queueing network models is often unrealistic, since in practice the finiteness of buffers may significantly affect the dynamics of the network. Finite buffers may yield pure loss behaviour (e.g.\ \cite{Kelly0229}), blocking behaviour (e.g.\ \cite{balsamo2001aqn}, \cite{dai1999heavy}) or overflow behaviour as in Chapter~8 of \cite{bookWolff1989}. It is this last case that is our focus. Many special cases of overflow networks have previously appeared in theory and applications. Explicit analysis of special structure networks is in \cite{hordijk1987stochastic}, \cite{Kaufmanetal1981}, \cite{Sendfeld2008}, \cite{vanDijk1987} and \cite{vanDijk1989}. Asymptotic analysis is in \cite{GurvichPerry2012}. Analysis connected with dedicated applications is in \cite{asaduzzaman2010loss}, \cite{KooleTalim2000} and \cite{Litvak2008}. However, all of these mentioned papers deal with either two queues only or with specialized topologies. Indeed there has not been a complete investigation of overflow networks with arbitrary routing and overflow topologies.

Nevertheless, overflow phenomena are omnipresent in applications. In health care operations, patients arriving to overcrowded facilities often overflow to other care centers. In city bike rental schemes, customers returning bikes to full centers often choose nearby centers. In amusement parks, attendees traverse between attraction rides, often choosing to not enter rides that are (momentarily) with big queues. In convention centers, delegates often traverse between exhibition booths and when one booth fills up, they overflow to alternative nearby booths. Other examples of multi-node systems with overflows also appear in urban road traffic, telecommunications, manufacturing and other domains.


Our key contributions are the following:
\begin{itemize}
	\item[(i)] We extend results about equation \eqref{eq:goodmanMasseyEq}, fully characterizing the class of models for which there is a unique non-negative solution.
	\item[(ii)] We build on these results to derive a sufficient condition for a unique non-negative solution for equation \eqref{eq:ourTrafficEquation}.
	\item[(iii)] We present a novel, efficient algorithm for solving equation \eqref{eq:ourTrafficEquation} and prove that it finds the solution under our sufficient condition. This algorithm terminates in at most $n^2$ steps, and generalizes the $n$-step algorithm of \cite{goodman1984nej} for \eqref{eq:goodmanMasseyEq}. Each step involves solving a linear set of equations of order at most $n$. An implementation of our algorithm, also allowing to reproduce the examples from this paper, is available at \cite{GitHub:overflow-algorithm}.
\end{itemize}

The remainder of this paper is structured as follows. In Section~\ref{sec:algMainRes} we present the main algorithms of this paper (Algorithm~\ref{alg:GMalgorithm} and Algorithm~\ref{alg:overflowAlgorithm}) together with the main results (Theorem~\ref{th:GMalgorithmworks} and Theorem~\ref{alg:overflowAlgorithm}). In Section~\ref{sec:proofs} we prove the main results. In Section~\ref{sec:numExample} we present example networks. We then conclude in Section~\ref{sec:conc}. The appendix contains a number of technical lemmas for networks without overflows. We provide an overview of the main notational conventions below.

\subsection*{Notation}

We use the following notation throughout.  Let $n \in \mathbb{N}$ and denote $N = \lbrace 1 , \dotsc , n \rbrace$. We assume that vectors are row vectors. A nonnegative vector is a vector whose entries are nonnegative. Inequalities of vectors are interpreted pointwise, so $x < y$ for two $n$-dimensional vectors $x$ and $y$ means that $x_{i} < y_{i}$ for all $i \in N$. If $x$ is a nonnegative $n$-dimensional vector, then $\abs{x} = \sum_{i \in N} x_{i}$. We also use the notation $\abs{{\cal A}}$ for the cardinality of a set ${\cal A}$. Suppose $P$ is an $n \times n$ matrix. The spectral radius of $P$ is $\specrad{P}$. We denote by $P^{k}$ the $k$-fold matrix product of $P$. If $P$ has entries $p_{ij}$, then $p_{ij}^{k}$ denote the corresponding entries of $P^{k}$. If $A, B \subset \lbrace 1 , \dotsc , d \rbrace$, then we define $P_{AB}$ as the matrix whose $(ij)$-th entry coincides with $p_{ij}$ if $i \in A$ and $j \in B$, and equals $0$ otherwise. In short, $P_{AB}$ is the matrix we obtain by zeroing out the rows $i \not\in A$ as well as the columns $j \not\in B$. We obtain the matrix $[P]_{AB}$ by erasing rows $i \not\in A$ and columns $j \not\in B$, so $[P]_{AB}$ has dimensions $\abs{A} \times \abs{B}$. We use the shorthand notation $P_{A} = P_{AN}$ and $[P]_{A} = [P]_{AN}$. If $P$ has nonnegative entries, then we call $P$ a nonnegative matrix. If $P$ is a nonnegative matrix, then we call $P$ stochastic if its row sums are all equal to $1$ and substochastic if its row sums are all smaller than or equal to $1$. The definition of communicating classes for substochastic matrices is the same as it is for stochastic matrices. That is, 
we say that two elements $i, j \in N$ communicate if $i = j$ or if $p_{ij}^{k} > 0$ and $p_{ji}^{l} > 0$ for some $k, l \ge 0$. 
This equivalence relationship divides $N$ into classes, which we call the communicating classes based on $P$. The incidence matrix of a nonnegative matrix $P$ has entries $0$ and $1$, where an entry equals $0$ if the corresponding entry of $P$ equals $0$ and equals $1$ if the corresponding entry of $P$ is strictly positive. We denote the identity matrix by $I$; its dimensions are always clear from the context. The vector with each of its entries equal to $1$ is simply denoted by $1$. The dimensions of the vector $1$ are also clear from the context.

\section{Algorithms and main results}
\label{sec:algMainRes}

We consider a fluid network of $n$ nodes and denote the set of nodes by $N = \lbrace 1 , \dotsc , n \rbrace$. The network is parametrized by a vector $\alpha \geq 0$ of exogenous arrival rates, a vector $\mu > 0$ of service capacities, a substochastic routing matrix $P$ with entries $p_{ij}$, and a substochastic overflow matrix $Q$ with entries $q_{ij}$. We denote the network by $( \alpha, \mu, P, Q )$. If there is no overflow possible ($Q = 0$), then we denote the network by $( \alpha, \mu, P )$ and refer to it as a {\em network without overflows}. Otherwise, if $q_{ij} > 0$ for at least one entry of $Q$, then we refer to it as an {\em overflow network}.


We first focus on the network $(\alpha, \mu, P)$ in which the overflow traffic equation \eqref{eq:ourTrafficEquation} reduces to \eqref{eq:goodmanMasseyEq}. We present the Goodman--Massey algorithm from \cite{goodman1984nej}, which solves \eqref{eq:goodmanMasseyEq} in $n$ iterations. The corresponding proof in \cite{goodman1984nej} is given under the so-called filled-or-drained ({\sc fd}) condition. We argue that this condition can be relaxed and call the relaxed condition the non-isolated (\textsc{ni}) condition. We show that the new \textsc{ni} condition encompasses all cases that admit a unique, nonnegative solution to \eqref{eq:goodmanMasseyEq} and we prove that the Goodman--Massey algorithm always finds the solution to \eqref{eq:goodmanMasseyEq} under this condition.

We then proceed to study the network $(\alpha, \mu, P, Q)$ and develop an algorithm that efficiently solves the corresponding overflow traffic equation \eqref{eq:ourTrafficEquation}. We prove that this new algorithm finds the solution of the overflow traffic equation \eqref{eq:ourTrafficEquation} under a condition that is derived from the \textsc{ni} condition and guarantees the existence and uniqueness of a nonnegative solution to \eqref{eq:ourTrafficEquation}.

\subsection*{Networks without overflows}
\label{subsec:solvingGoodmanMassey}

An important contribution of Goodman and Massey in \cite{goodman1984nej} is the introduction of an efficient algorithm (Algorithm~\ref{alg:GMalgorithm}) that solves the traffic equation \eqref{eq:goodmanMasseyEq}, provided that certain technical conditions are met. Every iteration requires solving a linear equation of the form
\begin{equation}
\label{eq:linearGMtrafficequation}
\lambda = \alpha + \mu P_{N \setminus S} + \lambda P_{S},
\quad
\mbox{for}
\quad
S \subset N.
\end{equation}

The idea behind the algorithm is quite intuitive. Consider the traffic equation~\eqref{eq:goodmanMasseyEq}. If we know beforehand which nodes are in the set of stable nodes $S$ (where $i \in S$ means that $\lambda_{i} < \mu_{i}$) and which nodes are in the set of nodes $N \setminus S$ that are not stable, then solving the traffic equation \eqref{eq:goodmanMasseyEq} reduces to solving the linear equation \eqref{eq:linearGMtrafficequation}. However, we do not know beforehand which nodes are stable. The algorithm deals with this iteratively by finding the stable nodes.

\begin{algorithm}[H]
	\caption{Goodman--Massey algorithm for solving the traffic equation \eqref{eq:goodmanMasseyEq}}
	\label{alg:GMalgorithm}
	\begin{algorithmic}
		\State $\kappa := 0$
		\State $S^{( \kappa )} := \emptyset$
		\Repeat
		\State $\kappa := \kappa + 1$
		\State Solve \eqref{eq:linearGMtrafficequation} based on $S := S^{( \kappa - 1)}$ to obtain a solution denoted by $\lambda^{( \kappa )}$
		\State $S^{( \kappa )} := \{ i \in N \, : \,  \lambda^{( \kappa )}_i < \mu_i \}$
		\Until $S^{( \kappa )} = S^{( \kappa - 1)}$
		\State $\lambda^{*} := \lambda^{( \kappa )}$
		\State \textbf{return} $\lambda^{*}$
	\end{algorithmic}
\end{algorithm}
The following intuition underlies the algorithm. Knowing that the maximum output rates of the nodes are given by the vector $\mu$, we start by considering the worst-case scenario in which the output rates of the nodes are given by the maximum output rates $\mu$, so essentially we assume that $S = \emptyset$. The corresponding solution of \eqref{eq:linearGMtrafficequation} is simply $\lambda^{(0)} = \alpha + \mu P$. Now there are two crucial observations. First, the vector $\lambda^{(0)}$ overestimates or equals the actual solution, since all nodes have their respective maximum output rate. Second, if a node $i$ is stable according to $\lambda^{(0)}$ (i.e., $\lambda^{(0)}_{i} < \mu_{i}$), then this node must be stable for all other scenarios with smaller output rates for the nodes. In particular, $i$ must be stable for the actual solution of the traffic equation \eqref{eq:goodmanMasseyEq}.
The estimate $\lambda^{(0)}$ gives us a set of stable nodes $S^{(0)}$. Knowing that these nodes must be stable for the actual solution of \eqref{eq:goodmanMasseyEq}, we compute the solution $\lambda^{(1)}$ of \eqref{eq:linearGMtrafficequation} based on $S = S^{(0)}$. Then $\lambda^{(1)}$ overestimates the actual solution of the traffic equation (because all nodes that might be unstable are given their respective maximum output rate), while $\lambda^{(1)}$ itself is smaller than $\lambda^{(0)}$. As before, this gives us a set of stable nodes $S^{(1)} \supset S^{(0)}$.
This procedure is repeated until $S^{(k+1)} = S^{(k)}$. In this case, the solution $\lambda^{(k+1)}$ of the linear equation \eqref{eq:linearGMtrafficequation} based on $S = S^{(k)}$ is also a solution of the traffic equation \eqref{eq:goodmanMasseyEq}. Consequently, the algorithm leads to a solution of the traffic equation \eqref{eq:goodmanMasseyEq} in at most $n$ iterations, where each iteration requires the solution of the linear equation \eqref{eq:linearGMtrafficequation}.

In \cite{goodman1984nej} the authors formulate a certain filled-or-drained ({\sc fd}) condition under which Algorithm~\ref{alg:GMalgorithm} finds the unique, nonnegative solution of traffic equation \eqref{eq:goodmanMasseyEq}. We introduce the more general non-isolated (\textsc{ni}) condition and prove that Algorithm~\ref{alg:GMalgorithm} also finds the solution to \eqref{eq:goodmanMasseyEq} under the \textsc{ni} condition. Importantly, we also show that the \textsc{ni} condition is equivalent to the traffic equation \eqref{eq:goodmanMasseyEq} having a unique, nonnegative solution.

To formulate the \textsc{ni} condition, we first divide the set of nodes into communicating classes according to the matrix $P$. We assume that there are exactly $m \in N$ different communicating classes corresponding to $P$ and we let $M = \lbrace 1 , \dotsc , m \rbrace$. Additionally, we denote the $m$ different classes by $C_{1} , \dotsc , C_{m}$. We characterize the classes as follows:
\begin{itemize}
	\item[(i)] A class $C$ can be \emph{filled} if $\alpha_{i} > 0$ for some $i \in C$, or if $\alpha_{i} > 0$ for some $i \in N$ and $p_{ij}^{k} > 0$ for some $j \in C$ and $k \in \mathbb{N}$.
	\item[(ii)] A class $C$ can be \emph{drained} if it can be \emph{externally drained} or \emph{internally drained}. We say that $C$ can be {externally drained} if there exists $i \in C$ such that $\sum_{j = 1}^{n} p_{ij} < 1$, so that fluid can escape from the network via node $i$. We say that $C$ can be {internally drained} if there exists $j \in N \setminus C$ such that $p_{ij} > 0$, so that fluid can escape to another class via node $i$.
	\item[(iii)] A class $C$ is called \emph{isolated} if it cannot be filled nor drained.
\end{itemize}
The {\sc fd} condition of \cite{goodman1984nej} requires that all nodes can be filled or externally drained. Our {\sc ni} condition is more general:
\begin{definition}
	\label{def:filleddrained}
	We say that a node $i$ is \textsc{ni} if it is not a node in an isolated class. Similarly, we say that the network $(\alpha, \mu, P)$ is \textsc{ni} if it does not contain an isolated class (all classes can be filled or drained).
\end{definition}
Our definition of a network being \textsc{ni} is more general than the \textsc{fd} condition, because in the \textsc{fd} condition drainable always means externally drainable. Although the \textsc{fd} condition is very elegant, we believe that our definition is not only more general but also more natural in light of the following results.

Because the output from the nodes is bounded by $\mu$, it can be shown using Tarski's fixed point theorem that the traffic equation \eqref{eq:goodmanMasseyEq} always has a solution, regardless of any condition, see \cite{chen1991dfn}. However, without any conditions imposed, the solution may be negative and may not make physical sense. It turns out that \textsc{ni} is equivalent to every solution of the traffic equation being nonnegative. Additionally, \textsc{ni} implies the uniqueness of a solution to \eqref{eq:goodmanMasseyEq}. This means that \textsc{ni} is actually equivalent to the existence of a unique, nonnegative solution to \eqref{eq:goodmanMasseyEq}. In view of this,
we use the following condition in the remainder of this paper.
\begin{condition}
	\label{as:filledordrained}
	The network $(\alpha, \mu, P)$ is \textsc{ni}.
\end{condition}

Under this condition, we show that Algorithm~\ref{alg:GMalgorithm} efficiently solves the traffic equation \eqref{eq:goodmanMasseyEq}. The proof is in the next section.

\begin{theorem}
	\label{th:GMalgorithmworks}
	The network $(\alpha, \mu, P)$ satisfies Condition \ref{as:filledordrained} if and only if the traffic equation \eqref{eq:goodmanMasseyEq} has a unique, nonnegative solution. Additionally, the Goodman--Massey algorithm presented in Algorithm~\ref{alg:GMalgorithm} finds this solution in at most $n$ iterations under Condition~\ref{as:filledordrained}. Each iteration consists of solving a linear equation in $n$ variables that has a unique solution.
\end{theorem}

We now move on to the main object of this paper.

\subsection*{Networks with overflows}
\label{subsec:solvingOverflow}

We consider the network $(\alpha, \mu, P, Q)$ with the aim of solving the corresponding overflow traffic equation~\eqref{eq:ourTrafficEquation}. As in the case without overflows, we give conditions under which a unique, nonnegative solution exists and we develop an algorithm, the {\em overflow algorithm}, that efficiently finds this solution in at most $n^{2}$ iterations.

The overflow algorithm, presented here as Algorithm \ref{alg:overflowAlgorithm}, consists of an inner loop and an outer loop. The inner loop is a variation of Algorithm~\ref{alg:GMalgorithm}. At its heart is the equation
\begin{align}
	\label{eq:innerLoopEq}
	\lambda = \alpha + (\lambda \wedge \mu) P_{A} + \mu P_{N \setminus A} + (\lambda - \mu) Q_{B},
	\quad
	\mbox{for}
	\quad
	A,B \subset N,
\end{align}
which we call the {\em partly linearized overflow equation} based on $A$ and $B$, since the $(\lambda - \mu)^{+}$ term in the overflow equation \eqref{eq:ourTrafficEquation} is replaced by $(\lambda - \mu)$. The inner loop solves the partly linearized overflow traffic equation for certain sets of nodes $A$ and $B$ by iteratively solving linear equations of the form
\begin{align}
	\label{eq:innerLoopLinear}
	\lambda = \alpha +   \mu P_{N \setminus S} + \lambda P_{S} +  (\lambda - \mu) Q_{B},
		\quad
\mbox{for}
\quad
S \subset A \subset N,\qquad B \subset N.
\end{align}
Thus, the inner loop is essentially Algorithm \ref{alg:GMalgorithm} based on the linear equation \eqref{eq:innerLoopLinear} rather than \eqref{eq:linearGMtrafficequation}. Each iteration of the outer loop selects the sets of nodes $A$ and $B$ that are used in the inner loop. Based on the inner loop, the outer loop then iteratively adds overflow streams to the equation. We present the algorithm first and then discuss the intuition and conditions behind it.

\begin{algorithm}[H]
	\caption{Algorithm for solving the overflow traffic equation \eqref{eq:ourTrafficEquation} -- {\em overflow algorithm}}
	\label{alg:overflowAlgorithm}
	\begin{algorithmic}
		\State $\kappa := 0$
		\State $U^{(\kappa)} := \emptyset$
		\Repeat
		\State $\kappa := \kappa + 1$
		\State $\ell := 0$
		\State $S^{(\ell)} := \emptyset$
		\Repeat
		\State $\ell := \ell + 1$
		\State Solve \eqref{eq:innerLoopLinear} based on $S := S^{(\ell - 1)}$ and $B := U^{(\kappa - 1)}$ to obtain a solution denoted by $\lambda^{(\kappa, \ell)}$
		\State $S^{(\ell)} := \lbrace i \in N \setminus U^{(\kappa - 1)} \, : \, \lambda^{(\kappa, \ell)}_{i} < \mu_{i} \rbrace$
		\Until $S^{(\ell)} = S^{(\ell - 1)}$
		\State $\lambda^{(\kappa)} := \lambda^{(\kappa, \ell)}$
		\State $U^{(\kappa)} := \lbrace i \in N \, : \, \lambda^{(\kappa)}_{i} \ge \mu_{i} \rbrace$
		\Until $U^{(\kappa)} = U^{(\kappa - 1)}$
		\State $\lambda^{\dagger} := \lambda^{(\kappa)}$
		\State \textbf{return} $\lambda^{\dagger}$
	\end{algorithmic}
\end{algorithm}
The overflow algorithm is based on the following idea.
We start by considering the network without any overflows. Then the overflow traffic equation \eqref{eq:ourTrafficEquation} reduces to \eqref{eq:innerLoopEq} with $A = N$ and $B = \emptyset$, and the Goodman--Massey algorithm (implicitly executed via the inner loop) finds the corresponding solution $\lambda^{(1)}$. Because overflows are not allowed, the solution $\lambda^{(1)}$ underestimates the actual solution to \eqref{eq:ourTrafficEquation} and thus an unstable node $i$ for $\lambda^{(1)}$ (meaning that $\lambda^{(1)}_{i} \geq \mu_{i}$) must be unstable for the solution to \eqref{eq:ourTrafficEquation}.

The first estimate $\lambda^{(1)}$ gives us a set of unstable nodes $U^{(1)}$. Knowing that these nodes must be unstable for the solution to \eqref{eq:ourTrafficEquation}, we allow the nodes in $U^{(1)}$ to overflow. Then the inner loop solves \eqref{eq:innerLoopEq} with $A = N \setminus U^{(1)}$ and $B = U^{(1)}$ and returns a solution $\lambda^{(2)}$. The crucial observation here is that $\lambda^{(2)}$ is lower bounded by $\lambda^{(1)}$. Indeed, we know that the nodes in $U^{(1)}$ are unstable for $\lambda^{(1)}$, which is the solution of the traffic equation if no overflows are allowed. Since we do allow the nodes $U^{(1)}$ to overflow when solving for $\lambda^{(2)}$, we know that $\lambda^{(2)}$ must be lower bounded by $\lambda^{(1)}$. In particular, the nodes in $U^{(1)}$ are unstable for $\lambda^{(2)}$ and $(\lambda^{(2)} - \mu) Q_{B} = (\lambda^{(2)} - \mu)^{+} Q_{B}$ for $B = U^{(1)}$. This also gives us a set of unstable nodes $U^{(2)} \supset U^{(1)}$ that are unstable for $\lambda^{(2)}$.

This procedure is repeated until $U^{(k + 1)} = U^{(k)}$. In this case, the solution $\lambda^{(k + 1)}$ of the partly linearized traffic equation \eqref{eq:innerLoopEq} based on $A = N \setminus U^{(k)}$ and $B = U^{(k)}$ satisfies $(\lambda^{(k + 1)} - \mu) Q_{B} = (\lambda^{(k + 1)} - \mu)^{+} Q_{B}$, because $U^{(k + 1)} = U^{(k)}$. Then $\lambda^{(k + 1)}$ also solves the overflow traffic equation \eqref{eq:ourTrafficEquation}. Consequently, the algorithm leads to a solution of \eqref{eq:ourTrafficEquation}. It stops after at most $n^{2}$ iterations, because the inner loop and the outer loop each stop after at most $n$ iterations.

The intuitive explanation of the algorithm above tacitly assumes that the overflow traffic equation \eqref{eq:ourTrafficEquation} has a (nonnegative) solution and that we can solve the related linear equations. The network $(\alpha, \mu, P, Q)$ has to satisfy certain conditions for this to be true. We formulate such conditions here and provide a short discussion. 

\begin{condition}
	\label{as:overflow}
	The network $(\alpha, \mu, P)$ satisfies Condition \ref{as:filledordrained}. Additionally, it holds that
	\begin{align}
		\label{eq:overflowCondition}
		\specrad{P_{A} + Q_{N \setminus A}} < 1,
	\end{align}
	for every set $A \subset N \setminus U^{*}$, where $U^{*} = \lbrace i \in N \, \vert \, \lambda^{*}_{i} \geq \mu_{i} \rbrace$ with $\lambda^{*}$ being the solution to \eqref{eq:goodmanMasseyEq}.
\end{condition}


The condition consists of two requirements. The first requirement is needed to guarantee that every solution to the overflow traffic equation \eqref{eq:ourTrafficEquation} is nonnegative. The need for the second requirement can be inferred from \eqref{eq:ourTrafficEquation} itself. Any nonnegative solution $\lambda$ to \eqref{eq:ourTrafficEquation} has a (possibly empty) set $A$ of stable nodes (so $A \subset N \setminus U^{*}$) and satisfies
\begin{align*}
\lambda = \alpha + \mu P_{N \setminus A} - \mu Q_{N \setminus A} + \lambda (P_{A} + Q_{N \setminus A}).
\end{align*}
Consequently, if we are to find $\lambda$, we have to solve this linear equation. The second requirement of Condition~\ref{as:overflow} ensures that this is possible and that the solution is unique. Note that since Condition~\ref{as:overflow} requires \eqref{eq:overflowCondition} to hold for at least $A = \emptyset$, the condition imposes that $\sigma(Q) < 1$.

Note that Condition \ref{as:overflow} holds under several natural, sufficient conditions. If $P 1 < 1$ and $Q 1 < 1$, then it holds automatically, because every eigenvalue must be strictly smaller than $1$. This corresponds to the scenario in which every node is leaking at least a little bit. In the important case that $Q = 0$, we find ourselves in the Goodman--Massey scenario in which no overflows are allowed. In this case Condition~\ref{as:filledordrained} and Condition~\ref{as:overflow} are equivalent: either both are true or both are false. This is a straightforward consequence of the proof of Theorem \ref{th:GMalgorithmworks}.

With Condition \ref{as:overflow} formulated we show that the overflow algorithm efficiently solves the overflow traffic equation \eqref{eq:ourTrafficEquation}.

\begin{theorem}
	\label{th:ourAlgorithmWorks}
	If the network $(\alpha, \mu, P, Q)$ satisfies Condition~\ref{as:overflow}, then the overflow traffic equation \eqref{eq:ourTrafficEquation} has a unique, nonnegative solution. Additionally Algorithm~\ref{alg:overflowAlgorithm} finds this solution in at most $n^{2}$ iterations. Each iteration consists of solving a linear equation in $n$ variables that has a unique solution.
\end{theorem}

It is important to note that the existence of a solution to \eqref{eq:ourTrafficEquation} under the conditions of Theorem \ref{th:ourAlgorithmWorks} actually follows from Algorithm \ref{alg:overflowAlgorithm} finding it. Thus, the proof that Algorithm \ref{alg:overflowAlgorithm} works under the specified conditions is also a proof of the existence of a solution to \eqref{eq:ourTrafficEquation}. This is in sharp contrast with Theorem~\ref{th:GMalgorithmworks}, where the existence of a solution does not require any conditions on the network and is a simple consequence of Tarski's fixed point theorem.

\section{Proofs}
\label{sec:proofs}

Some parts of the proofs for the main results (Theorem~\ref{th:GMalgorithmworks} and Theorem~\ref{th:ourAlgorithmWorks}) rely on lemmas that are presented in the appendix. We begin by proving Theorem~\ref{th:GMalgorithmworks} and break the proof into two separate components. For presentation purposes, we first prove that the Goodman--Massey algorithm presented in Algorithm~\ref{alg:GMalgorithm} finds a solution. We then separately move on to prove existence and uniqueness.


\begin{proof} ({\em The Goodman--Massey algorithm finds a solution.)}\\
We first prove that the Goodman--Massey algorithm finds a solution to \eqref{eq:goodmanMasseyEq} in at most $n$ iterations under Condition~\ref{as:filledordrained}. The proof is an inductive proof; for readability, we spell out the first induction step.

	In the first iteration of the algorithm, we take the nonnegative vector $\lambda^{(1)} = \alpha + \mu P$ as the first candidate for the solution. If $\lambda^{(1)} \geq \mu$, then it is actually a solution and the algorithm terminates.
	
	Suppose that the first iteration does not give a solution. Then the second iteration constructs a new candidate for the solution as follows.
	Because $\lambda^{(1)}$ was not a solution, there must be at least one $\lambda^{(1)}_{i} < \mu_{i}$. Define the set of stable nodes
	\begin{align*}
		S^{(1)} &= \lbrace i \, \vert \, \lambda^{(1)}_{i} < \mu_{i} \rbrace.
	\end{align*}
	We show later that these nodes are guaranteed to be stable, in the sense that $\lambda_{i} < \mu_{i}$ for the actual solution $\lambda$ if $i \in S^{(1)}$. Now let $S = S^{(1)}$ and define the pairwise disjoint sets $D_{k} = C_{k} \cap S$ where $C_k,~ k=1,\ldots,m$, are the communicating classes of $P$. We would like to invoke Lemma~\ref{lem:triangularblockspecrad} to show that $\specrad{P_{SS}} < 1$, so we have to verify that each matrix $P_{D_{k} D_{k}}$ satisfies $\specrad{P_{D_{k} D_{k}}} < 1$. To this end, we distinguish the following two cases:
	\begin{itemize}
		\item  Suppose that $C_{k}$ can be drained. Then $\specrad{P_{C_{k} C_{k}}} < 1$ and thus $\specrad{P_{D_{k} D_{k}}} < 1$, because $D_{k} \subset C_{k}$.

		\item Suppose that $C_{k}$ cannot be drained. Then it can be filled, due to the \textsc{ni} condition. Setting $\gamma = \lambda^{(1)}$, we note that $\gamma$ is nonnegative and satisfies
		\[
		\gamma = \alpha + \mu P \ge \alpha + (\gamma \wedge \mu) P.
		\]
		Now invoking Lemma \ref{lem:positivekappa}, we conclude that $\gamma_i > \mu_i$ for some $i \in C_{k}$ and hence that $D_{k} \not= C_{k}$ (it is a proper subset). Further since $C_{k}$ is a communicating class that cannot be drained, $P_{C_{k} C_{k}}$ is stochastic. It thus follows that $\specrad{P_{D_{k} D_{k}}} < 1$ because $D_k$ is a strict subset of $C_k$.
\end{itemize}
Therefore, based on these two cases, we conclude from Lemma \ref{lem:triangularblockspecrad} that $\specrad{P_{SS}} < 1$, which implies that $I - P_{SS}$ is invertible.
	
Now set $T = N \setminus S$ and use our (at this point still intuitive) knowledge that the nodes in $S$ must be stable to define a new candidate solution $\lambda^{(2)}$ via \eqref{eq:linearGMtrafficequation}. In this case, $\lambda^{(2)}$ satisfies
	\begin{align*}
		[\lambda^{(2)}]_{S} &= [\alpha]_{S} + [\mu]_{T} [P]_{TS} + [\lambda^{(2)}]_{S}  \, [P]_{SS}, \\
		[\lambda^{(2)}]_{T} &= [\alpha]_{T} + [\mu]_{T} P_{T \, T} + [\lambda^{(2)}]_{S} \, [P]_{ST},
	\end{align*}
	and thus
	\begin{align}
		[\lambda^{(2)}]_{S} &= ( [\alpha]_{S} + [\mu]_{T} [P]_{TS} ) ([I]_{SS} - [P]_{SS})^{-1} \leq \lambda^{(1)}_{S}, \label{eq:iterlambdaineq} \\
		[\lambda^{(2)}]_{T} &= [\alpha]_{T} + [\mu]_{T} P_{T \, T} + [\lambda^{(2)}]_{S} [P]_{ST}. \nonumber
	\end{align}
In particular, $\lambda^{(2)}$ exists, and is nonnegative and unique. The existence of the inverse is a direct consequence of $\specrad{P_{SS}} < 1$, while the inequality follows from the definition of $\lambda^{(1)}$ and $S$.
	
Now illustrating the next step, we define the set of stable nodes corresponding to $\lambda^{(2)}$ via
	\begin{align*}
		S^{(2)} &= \lbrace i \, \vert \, \lambda^{(2)}_{i} < \mu_{i} \rbrace.
	\end{align*}
	Then $S^{(2)} \supset S^{(1)}$ due to \eqref{eq:iterlambdaineq}, which confirms our intuition that the nodes in $S^{(1)}$ must be stable. Note that if $S^{(2)} = S^{(1)}$, then $\lambda^{(2)}$ solves the traffic equation~\eqref{eq:goodmanMasseyEq} and we are done.
	
Suppose that the first two steps do not give us a solution of the traffic equation. In this step, we generalize the second step to an induction argument as follows. Suppose that for fixed $\kappa \in \mathbb{N}$ we have a nonnegative vector $\lambda^{(\kappa)}$ and a corresponding set of (stable) nodes $S^{(\kappa)} \subset N$ satisfying the following properties. Using the shorthand notation $x = \lambda^{(\kappa)}$, $S = S^{(\kappa)}$, and $T = N \setminus S$ as previously, it holds that $S = \lbrace x_{i} < \mu_{i} \rbrace$ and
\[
		x \geq \alpha + \mu_{T} P_{T N} + x_{S} P_{S N}.
\]
We define a new candidate solution vector $\lambda^{(\kappa + 1)} = y$, where $y$ is the solution to the linear equation
	\begin{align*}
		y &= \alpha + \mu_{T} P_{T N} + y_{S} P_{S N},
	\end{align*}
	which is equivalent to
	\begin{align*}
		[y]_{S} &= [\alpha]_{S} + [\mu]_{T} [P]_{T S} + [y]_{S} [P]_{S S}, \\
		[y]_{T} &= [\alpha]_{T} + [\mu]_{T} [P]_{T T} + [y]_{S} [P]_{S T}.
	\end{align*}
We need to  show that this actually gives us a unique, nonnegative vector $y$. The key step is the observation that $\specrad{[P]_{SS}} < 1$ or, equivalently, that $\specrad{P_{SS}} < 1$. Indeed, we have $S = \lbrace x_{i} < \mu_{i} \rbrace$ by assumption, and
\[
x \geq \alpha + \mu_{T} P_{T N} + x_{S} P_{S N} = \alpha + (x \wedge \mu) P,
\]
	so as before Lemma \ref{lem:positivekappa} implies that $\specrad{P_{S S}} < 1$. It follows that
	\begin{align*}
		[y]_{S} &= ([\alpha]_{S} + [\mu]_{T} [P]_{T S}) ([I]_{S S} - [P]_{S S})^{-1}, \\
		[y]_{T} &= [\alpha]_{T} + [\mu]_{T} [P]_{T T} + [y]_{S} [P]_{S T} + [\mu]_{T} [P]_{T T}.
	\end{align*}
	The matrix $([I]_{SS} - [P]_{SS})^{-1}$ is positive because $\specrad{[P]_{SS}} < 1$, so the solution $y$ is indeed unique and nonnegative.
	
	This gives us a nonnegative candidate solution $\lambda^{(\kappa + 1)} = y$ and we define the corresponding set of stable nodes $S^{(\kappa + 1)} = \lbrace y_{i} < \mu_{i} \rbrace$. Now there are two crucial observations. The first observation is that
	\begin{align*}
		[x]_{S} ([I]_{SS} - [P]_{SS}) &\geq [\alpha]_{S} + [\mu]_{T} [P]_{TS} = [y]_{S} ([I]_{SS} - [P]_{SS}).
	\end{align*}
	Because $([I]_{SS} - [P]_{SS})^{-1}$ is positive, this implies that $[x]_{S} \geq [y]_{S}$ and thus $S^{(\kappa + 1)} \supset S^{(\kappa)}$. The second observation is that $\lambda^{(\kappa + 1)} = y$ actually solves the traffic equation if $S^{(\kappa + 1)} = S^{(\kappa)}$. A simple induction argument then completes the proof.
\end{proof}

In the previous proof, we demonstrated that the Goodman--Massey algorithm finds a nonnegative solution to \eqref{eq:goodmanMasseyEq} under Condition \ref{as:filledordrained}. We now argue that this is the unique solution to \eqref{eq:goodmanMasseyEq} under this condition. We also prove the converse: if there exists a unique, nonnegative solution to \eqref{eq:goodmanMasseyEq}, then Condition \ref{as:filledordrained} holds.

\begin{proof} ({\em The network $(\alpha, \mu, P)$ satisfies Condition \ref{as:filledordrained} if and only if the traffic equation \eqref{eq:goodmanMasseyEq} has a unique, nonnegative solution.})\\
We first show that if the network $(\alpha, \mu, P)$ is \textsc{ni}, then there exists at most one solution to the traffic equation \eqref{eq:goodmanMasseyEq}. For this let $x$ and $y$ both be solutions to the traffic equation. Lemma \ref{lem:fdnonnegsolposfil} implies that $x$ and $y$ are nonnegative under the present conditions.
	
	Denote $F = \lbrace i \in N \, \vert \, x_{i} > y_{i} \rbrace$ and assume that $F$ is nonempty. Without loss of generality this is the negation assumption. Because $x$ and $y$ are both solutions to the traffic equation \eqref{eq:goodmanMasseyEq}, we get
	\begin{align*}
		x_{i} - y_{i}
		&= \sum_{j \in N} (x_{j} \wedge \mu_{j} - y_{j} \wedge \mu_{j}) p_{ji}
	\end{align*}
	for all $i \in N$. Summing over $F$ gives us
\begin{equation}
\label{eq:sumOfF}
		\sum_{i \in F} (x_{i} - y_{i})
		= \sum_{j \in N} \big( \sum_{i \in F} p_{ji} \big) (x_{j} \wedge \mu_{j} - y_{j} \wedge \mu_{j})
		\leq \sum_{j \in F} \big( \sum_{i \in F} p_{ji} \big) (x_{j} \wedge \mu_{j} - y_{j} \wedge \mu_{j}).
\end{equation}
The inequality follows from the fact that $x_i \leq y_i$ for $i \in N \setminus F$. This can be rewritten as
\[
\sum_{i \in F} (x_{i} - y_{i})  \le \sum_{i \in F} \eta_i (x_{i} \wedge \mu_{i} - y_{i} \wedge \mu_{i}),
\]
where $\eta_i = \sum_{j \in F} p_{ij}$. Since $x_i \ge y_i$ for $i \in F$, it follows that
\begin{equation}
\label{eq:singleElementIneq}
x_i - y_i \ge x_{i} \wedge \mu_{i} - y_{i} \wedge \mu_{i} \ge \eta_i (x_{i} \wedge \mu_{i} - y_{i} \wedge \mu_{i} ).
\end{equation}
Combining \eqref{eq:sumOfF} and \eqref{eq:singleElementIneq}, we see that both inequalities in \eqref{eq:singleElementIneq} must be equalities for all $i \in F$ and hence
\[
\eta_i = \sum_{j \in F} p_{ij} = 1
\qquad
\text{for all}
\qquad
i \in F.
\]
This implies that $[P]_{FF}$ is a stochastic matrix, so there exists a class $C$ of $P$  with $C \subset F$ where $C$ cannot be drained. We now distinguish two cases.

If $C$ cannot be filled, then Lemma \ref{lem:fdnonnegsolposfil} implies that $x_{i} = y_{i} = 0$ for all $i \in C$, which contradicts the definition of $F$. If $C$ can be filled, then Lemma \ref{lem:positivekappa} implies that $x_{i} > \mu_{i}$ for some $i \in C$, because the network is \textsc{ni}. However, this is contradicted by the first inequality in \eqref{eq:singleElementIneq}, which implies that $x_{i} \leq \mu_{i}$ for all $i \in C$. We conclude from these contradictions that $F$ must be empty and thus $x = y$.

We now show that if the network $(\alpha, \mu, P)$ is not \textsc{ni}, then the traffic equation \eqref{eq:goodmanMasseyEq} admits a solution $\lambda$ that is not nonnegative. This means that the network $(\alpha, \mu, P)$ is \textsc{ni} if  every solution $\lambda$ to the traffic equation \eqref{eq:goodmanMasseyEq} is nonnegative.

Because the network is not \textsc{ni}, there exists at least one isolated class, say $C$. The class $C$ being isolated implies that \eqref{eq:goodmanMasseyEq} restricted to $C$ can be represented as
	\begin{align*}
		[\lambda]_{C} = ([\lambda]_{C} \wedge [\mu]_{C}) [P]_{CC},
	\end{align*}
	with $[P]_{CC}$ a stochastic, irreducible matrix. Then $[P]_{CC}$ is the transition probability matrix of an irreducible Markov chain with equilibrium distribution $\pi^{(C)} > 0$. In particular, $[\lambda]_{C} = -\pi^{(C)} < 0$ is a solution of the equation above.
	
	Let $D$ be the collection of all nodes that are not isolated. Then the network $([\alpha]_{D} , [\mu]_{D} , [P]_{DD})$ is \textsc{ni} with corresponding traffic equation
	\begin{align*}
		[\lambda]_{D} = [\alpha]_{D} + ([\lambda]_{D} \wedge [\mu]_{D}) [P]_{DD}.
	\end{align*}
	Lemma \ref{lem:ourBoyTarski} guarantees the existence of a nonnegative solution $[\lambda]_{D}$.
	
	Consequently, the traffic equation associated with $(\alpha, \mu, P)$ admits a solution $\lambda$ that is not nonnegative, where $[\lambda]_{C} = -\pi^{(C)}$ for every isolated class $C$ and $[\lambda]_{D}$ is a nonnegative solution to the traffic equation associated with the network $([\alpha]_{D} , [\mu]_{D} , [P]_{DD})$ based on the nodes that are not isolated.
\end{proof}

We now advance to prove Theorem~\ref{th:ourAlgorithmWorks}, for which we need the following two lemmas.
	
\begin{lemma}
	\label{lem:biggerSolQ}
	Consider the network $(\alpha, \mu, P, Q)$. Assume that $(\alpha, \mu, P)$ satisfies Condition \ref{as:filledordrained} and denote by $\lambda^{*}$ the unique, nonnegative solution to the traffic equation \eqref{eq:goodmanMasseyEq} associated with $(\alpha, \mu, P)$. If $\lambda$ is a solution to the overflow traffic equation \eqref{eq:ourTrafficEquation} associated with $(\alpha, \mu, P, Q)$, then $\lambda \geq \lambda^{*} \geq 0$.
\end{lemma}
\begin{proof}
Let $\lambda$ be a solution to the overflow equation. Then $\eps =  (\lambda - \mu)^{+} Q$ is nonnegative and
	\begin{align*}
		\lambda = \alpha + \eps + (\lambda \wedge \mu) P.
	\end{align*}
	Because $(\alpha, \mu, P)$ is \textsc{ni}, Lemma \ref{lem:goodmanmasseyordered} implies that $\lambda \geq \lambda^{*}$.
\end{proof}

\begin{lemma}
	\label{lem:nonnegativeanduniqueoverflowsolution}
	Consider the network $(\alpha, \mu, P, Q)$. If it satisfies Condition \ref{as:overflow} then there exists at most one solution to the overflow equation associated with $(\alpha, \mu, P, Q)$.
\end{lemma}

\begin{proof}
Let $x$ and $y$ both be solutions to the overflow traffic equation \eqref{eq:ourTrafficEquation}. Lemma~\ref{lem:biggerSolQ} implies that $x \geq \lambda^{*}$ and $y \geq \lambda^{*}$ under the present conditions, so both $x$ and $y$ are nonnegative.
	
Denote $F = \lbrace i \in N \, \vert \, x_{i} > y_{i} \rbrace$ and assume that $F$ is nonempty. Without loss of generality this is the negation assumption. Because $x$ and $y$ are both solutions to the overflow equation, we get
	\begin{align*}
	x_{i} - y_{i} &= \sum_{j \in N} (x_{j} \wedge \mu_{j} - y_{j} \wedge \mu_{j}) p_{ji} + \sum_{j \in N} ( (x_{j} - \mu_{j})^{+} - (y_{j} - \mu_{j})^{+} ) q_{ji}\\
	&\le  \sum_{j \in F} (x_{j} \wedge \mu_{j} - y_{j} \wedge \mu_{j}) p_{ji} + \sum_{j \in F} ( (x_{j} - \mu_{j})^{+} - (y_{j} - \mu_{j})^{+} ) q_{ji},
	\end{align*}
	for all $i \in N$. Summing over $F = \lbrace i \in N \, \vert \, x_{i} > y_{i} \rbrace$ gives us
	\begin{align}
	\begin{split}
	\sum_{i \in F} (x_{i} - y_{i})
	&\leq \sum_{j \in F} \big( \sum_{i \in F} p_{ji} \big) (x_{j} \wedge \mu_{j} - y_{j} \wedge \mu_{j})
	 + \sum_{j \in F} \big( \sum_{i \in F} q_{ji} \big) ((x_{j} - \mu_{j})^{+} - (y_{j} - \mu_{j})^{+})\\
	 &=  \sum_{i \in F} \big( \sum_{j \in F} p_{ij} \big) (x_{i} \wedge \mu_{i} - y_{i} \wedge \mu_{i})
	 + \sum_{i \in F} \big( \sum_{j \in F} q_{ij} \big) ((x_{i} - \mu_{i})^{+} - (y_{i} - \mu_{i})^{+}).
	\end{split}
	\end{align}
where in the second equation we just re-write with $i$ and $j$ flipped.
	
Now partition $F$ into the sets $F_{1} = \lbrace i \in F \, \vert \, y_{i} \geq \mu_{i} \rbrace$, $F_{2} = \lbrace i \in F \, \vert \, x_{i} < \mu_{i} \rbrace$, and $F_{3} = \lbrace i \in F \, \vert \, x_{i} \geq \mu_{i}, y_{i} < \mu_{i} \rbrace$. Then we can rewrite the right-hand side of the equation above to obtain
	\begin{align*}
	\begin{split}
	\sum_{i \in F} (x_{i} - y_{i})
	&\leq \sum_{i \in F_{1}} \big( \sum_{j \in F} q_{ij} \big) ((x_{i} - \mu_{i})^{+} - (y_{i} - \mu_{i})^{+}) \\
	&\phantom{{} \leq {}} {} + \sum_{i \in F_{2}} \big( \sum_{j \in F} p_{ij} \big) (x_{i} - y_{i}) \\
	&\phantom{{} \leq {}} {} + \sum_{i \in F_{3}} \bigg( \big( \sum_{j \in F} p_{ij} \big) (\mu_{i} - y_{i}) + \big( \sum_{j \in F} q_{ij} \big) (x_{i} - \mu_{i})^{+} \bigg)\\
	&\leq \sum_{i \in F_{1}} \big( \sum_{j \in F} q_{ij} \big) (x_{i} - y_{i}) \\
	&\phantom{{} \leq {}} {} + \sum_{i \in F_{2}} \big( \sum_{j \in F} p_{ij} \big) (x_{i} - y_{i}) \\
	&\phantom{{} \leq {}} {} + \sum_{i \in F_{3}} \bigg( \big( \sum_{j \in F} p_{ij} \big) (\mu_{i} - y_{i}) + \big( \sum_{j \in F} q_{ij} \big) (x_{i} - \mu_{i})^{+} \bigg).
	\end{split}
	\end{align*}
Because $x > y$ on $F$, it follows from this inequality that
	\begin{align}
	\label{eq:stochasticmatrixequalities}
	\begin{split}
	\sum_{j \in F} q_{ij} = 1 \text{ for all } i \in F_{1}, \\
	\sum_{j \in F} p_{ij} = 1 \text{ for all } i \in F_{2}, \\
	\sum_{j \in F} p_{ij} = 1 \text{ for all } i \in F_{3}.
	\end{split}
	\end{align}
	
	We now argue that this contradicts Condition \ref{as:overflow}. To do so, we can create a $|F| \times |F|$ matrix $R$, using the elements $q_{ij}$ for $i \in F_{1}$ and $p_{ij}$ for $i \in F_{2} \cup F_{3}$. From \eqref{eq:stochasticmatrixequalities}, $R$ is stochastic and hence $\specrad{R} = 1$. Further $\specrad{P_{A} + Q_{N \setminus A}} = 1$ for any matrix $P_{A} + Q_{N \setminus A}$ that contains $R$ as a submatrix. We now create such a matrix by defining $A = N \setminus (F_{1} \cup U^*)$.
	
	Now recall that $x \geq \lambda^{*}$ and $y \geq \lambda^{*}$, and that $U^* = \lbrace i \in N \, \vert \, \lambda^{*}_{i} \geq \mu_{i} \rbrace$. It follows that $F_{2} \subset N \setminus U^*$ and $F_{3} \subset N \setminus U^*$. Also observe that $F_{2} \cup F_{3} \subset A$. This implies that we can represent $R$ as
	\begin{align*}
		R &= [P_{A} + Q_{N \setminus A}]_{FF}.
	\end{align*}
However, since $A \subset N \setminus U^*$, Condition \ref{as:overflow} stipulates that $\specrad{P_{A} + Q_{N \setminus A}} < 1$. This is a contradiction, because $\specrad{R} = 1$.
\end{proof}

We are now in a position to prove Theorem~\ref{th:ourAlgorithmWorks}.

\begin{proof}
Lemma~\ref{lem:biggerSolQ} implies that every solution is nonnegative, while Lemma~\ref{lem:nonnegativeanduniqueoverflowsolution} implies that there is at most one solution. To complete the proof, we now constructively show that Algorithm~\ref{alg:overflowAlgorithm} finds a solution  under Condition~\ref{as:overflow}.

	The algorithm starts by defining $U^{(0)} = \emptyset$. The first iteration of the outer loop is equivalent to the Goodman--Massey algorithm. Because we assume that $(\alpha, \mu, P)$ is \textsc{ni}, the first iteration of the outer loop yields the unique, nonnegative vector $\lambda^{(1)}$ satisfying
\begin{equation}
\label{eq:GMinOverflowAlgSol1}
\lambda^{(1)} = \alpha + (\lambda^{(1)} \wedge \mu) P
\end{equation}
as a first candidate for the solution. This specifies the set of unstable, overflowing nodes $U^{(1)} = \lbrace i \in N \, \vert \, \lambda^{(1)}_{i} \ge \mu_{i} \rbrace$, which coincides with $U^*$.
	If $U^{(1)}$ is empty, then $\lambda^{(1)}$ is actually a solution to the overflow traffic equation and the algorithm terminates.
	In case the first iteration does not lead to a solution, the algorithm performs another iteration of the outer loop and generates a new candidate solution $\lambda^{(2)}$.
	
	The following induction argument shows that the algorithm terminates after at most $n$ iterations of the outer loop and yields the solution of the overflow traffic equation \eqref{eq:ourTrafficEquation}.
	
	Consider iteration $\kappa \geq 1$ of the outer loop. Our induction hypothesis is that iteration $\kappa$ yields $\lambda^{(\kappa)}$ satisfying the partly linearized overflow equation \eqref{eq:innerLoopEq} with $A = N \setminus U^{(\kappa)}$ and $B = U^{(\kappa - 1)}$:
	\begin{align}
		\label{eq:inHypothesisPartlyLinearized}
		\lambda^{(\kappa)} = \alpha + (\lambda^{(\kappa)} \wedge \mu) P_{N \setminus U^{(\kappa)}} + \mu P_{U^{(\kappa)}} + (\lambda^{(\kappa)} - \mu) Q_{U^{(\kappa - 1)}},
	\end{align}
	where $U^{(\kappa)} = \lbrace i \in N \, \vert \, \lambda^{(\kappa)}_{i} \geq \mu_{i} \rbrace$ as specified in the algorithm. A second part of our induction hypothesis is that $U^{(\kappa - 1)} \subset U^{(\kappa)}$ and $U^* \subset U^{(\kappa)}$.
	
	We start by deriving some properties of $\lambda^{(\kappa)}$. From \eqref{eq:inHypothesisPartlyLinearized} it follows that
	\begin{align*}
		\lambda^{(\kappa)} = \alpha + \lambda^{(\kappa)} P_{N \setminus U^{(\kappa)}} + \mu P_{U^{(\kappa)}} + (\lambda^{(\kappa)} - \mu) Q_{U^{(\kappa - 1)}}.
	\end{align*}
	The induction assumption $U^{(\kappa - 1)} \subset U^{(\kappa)}$ together with nonnegativity of the vector $[\lambda^{(\kappa)} - \mu]_{U^{(\kappa)}}$ and the matrix $Q$ then implies that
	\begin{align*}
		\lambda^{(\kappa)} \leq \alpha + \lambda^{(\kappa)} P_{N \setminus U^{(\kappa)}} + \mu P_{U^{(\kappa)}} + (\lambda^{(\kappa)} - \mu) Q_{U^{(\kappa)}}.
	\end{align*}
	Now let $L \subset N \setminus U^{(\kappa)}$. Based on the previous inequality, we obtain that
	\begin{align*}
		\lambda^{(\kappa)} \leq \alpha + \lambda^{(\kappa)} P_{L} + \mu P_{(N \setminus U^{(\kappa)}) \setminus L} + \mu P_{U^{(\kappa)}} + (\lambda^{(\kappa)} - \mu) Q_{U^{(\kappa)}}.
	\end{align*}
	Since $N \setminus U^{(\kappa)} \subset N \setminus U^*$ as an immediate result of the induction assumption, Condition~\ref{as:overflow} guarantees that $\specrad{P_{N \setminus U^{(\kappa)}} + Q_{U^{(\kappa)}}} < 1$. With $L$ being a subset of $N \setminus U^{(\kappa)}$, we also have
	\begin{align}
		\label{eq:spLessQ1}
		\specrad{P_{L} + Q_{U^{(\kappa)}}} < 1.
	\end{align}
	Consequently, $(I - P_{L} - Q_{U^{(\kappa)}} )^{-1}$ exists and is nonnegative, so
	\begin{align}
		\label{eq:Lsubsetinequality}
		\lambda^{(\kappa)} \leq ( \alpha + \mu P_{(N \setminus U^{(\kappa)}) \setminus L} + \mu P_{U^{(\kappa)}} - \mu Q_{U^{(\kappa)}} ) (I - P_{L} - Q_{U^{(\kappa)}} )^{-1}
	\end{align}
	for every set of nodes $L \subset N \setminus U^{(\kappa)}$.
	
	We now proceed with the induction step. We show that, given the results of iteration $\kappa$ of the outer loop, iteration $\kappa + 1$ finds $\lambda^{(\kappa + 1)}$ satisfying the partly linearized overflow equation
	\begin{align}
		\label{eq:kappaPlusOneOuterLoop}
		\lambda^{(\kappa + 1)} = \alpha + (\lambda^{(\kappa + 1)} \wedge \mu) P_{N \setminus U^{(\kappa + 1)}} + \mu P_{U^{(\kappa + 1)}} + (\lambda^{(\kappa + 1)} - \mu) Q_{U^{(\kappa)}},
	\end{align}
	where $U^{(\kappa + 1)} = \lbrace i \in N \, \vert \, \lambda^{(\kappa + 1)}_{i} \geq \mu_{i} \rbrace$ as specified in the algorithm. To show this, we analyze the inner loop and prove that it yields the desired solution,  $\lambda^{(\kappa + 1)}$.
	
	The key property here is \eqref{eq:spLessQ1}, which holds for every $L \subset N \setminus U^{(\kappa)}$. Observe that in iteration $\ell+1$ of the inner loop, we consider $L = S^{(\ell)} \subset N \setminus U^{(\kappa)}$ and obtain a candidate solution of \eqref{eq:innerLoopLinear} that has the form
	\begin{align}
	\label{eq:ellInnerEquation}
		\lambda^{(\kappa + 1 , \ell+1)} = \alpha + \lambda^{(\kappa + 1 , \ell+1)} P_{L} + \mu P_{(N \setminus U^{(\kappa)}) \setminus L} + \mu P_{U^{(\kappa)}} + (\lambda^{(\kappa + 1, \ell+1)} - \mu) Q_{U^{(\kappa)}}.
	\end{align}
	Now using \eqref{eq:spLessQ1}, similar arguments as in the proof of Theorem~\ref{th:GMalgorithmworks} establish that $\lambda^{(\kappa + 1 , \ell + 1)} \leq \lambda^{(\kappa + 1 , \ell)}$, and that the inner loop terminates after at most $n$ iterations and returns the solution $\lambda^{(\kappa + 1)}$ satisfying \eqref{eq:kappaPlusOneOuterLoop}.
	
	Moreover, we know from \eqref{eq:Lsubsetinequality} and \eqref{eq:ellInnerEquation} that $\lambda^{(\kappa + 1 , \ell)} \geq \lambda^{(\kappa)}$ for every iteration $\ell$ of the inner loop. This holds in particular for the final iteration, so $\lambda^{(\kappa + 1)} \geq \lambda^{(\kappa)}$. This implies that $U^{(\kappa)} \subset U^{(\kappa + 1)}$ and that $U^* \subset U^{(\kappa + 1)}$. Thus the induction hypothesis is also satisfied for $\kappa + 1$.
	
	The base of the induction is satisfied due to \eqref{eq:GMinOverflowAlgSol1} and the fact that $U^{(1)} = U^*$. Hence, the induction hypothesis holds for all $\kappa \geq 1$.
	
	The algorithm terminates after at most $n + 1$ iterations of the outer loop, because $U^{(\kappa)} \subset U^{(\kappa + 1)}$. Note that $\lambda^{(\kappa + 1)}$ is a solution of the overflow traffic equation~\eqref{eq:ourTrafficEquation} if $U^{(\kappa)} = U^{(\kappa + 1)}$ and thus the algorithm terminates. This shows that the algorithm terminates after finitely many iterations and finds the solution to \eqref{eq:ourTrafficEquation}.

	We end this proof by showing that an upper bound for the total number of iterations of the inner loop of the algorithm is given by
	\begin{equation}
		\label{eq:maxNumIterations}
		1 + \frac{n(n + 1)}{2}.
	\end{equation}
	In a worst case scenario, all nodes are overflowing, the outer loop adds them one by one, and the inner loop iterates over all potentially stable nodes until finding the new overflowing node that is to be added in the outer loop. In this case, the first iteration of the outer loop finds one overflowing node and the inner loop needs $n$ iterations to find it. Then the second iteration of the outer loop finds the second overflowing node and the inner loop iterates over the remaining $n - 1$ potentially stable nodes to find it. This continues until all nodes turn out to be overflowing, which takes $\sum_{i = 1}^{n} (n + 1 - i) = n (n + 1) / 2$ iterations of the inner loop in this scenario. In its current formulation, Algorithm~\ref{alg:overflowAlgorithm} also needs one additional iteration to determine that a solution is found. This leads to the upper bound \eqref{eq:maxNumIterations}. 	
	Hence, the total number of iterations is bounded by $n^2$ for every network with $n \geq 2$ nodes.
	
	
\end{proof}

\section{Example Networks}
\label{sec:numExample}

We now provide four examples for illustrating the computational ability of our algorithm as well as for focusing on some subtle aspects. We first demonstrate the effectiveness of the overflow algorithm on a structured family of networks in Example~1. This is followed by Example~2  where we present a worst-case scenario in terms of running time of the algorithm.  Example~3  helps understand the difference between the \textsc{fd} condition from \cite{goodman1984nej} and our \textsc{ni} condition (Condition~\ref{as:filledordrained}). We then close with Example~4 presenting a scenario in which Condition~\ref{as:overflow} is not satisfied and there is no unique solution to the overflow traffic equation~\eqref{eq:ourTrafficEquation}.

	\begin{figure}
		\includegraphics[width=.5\textwidth]{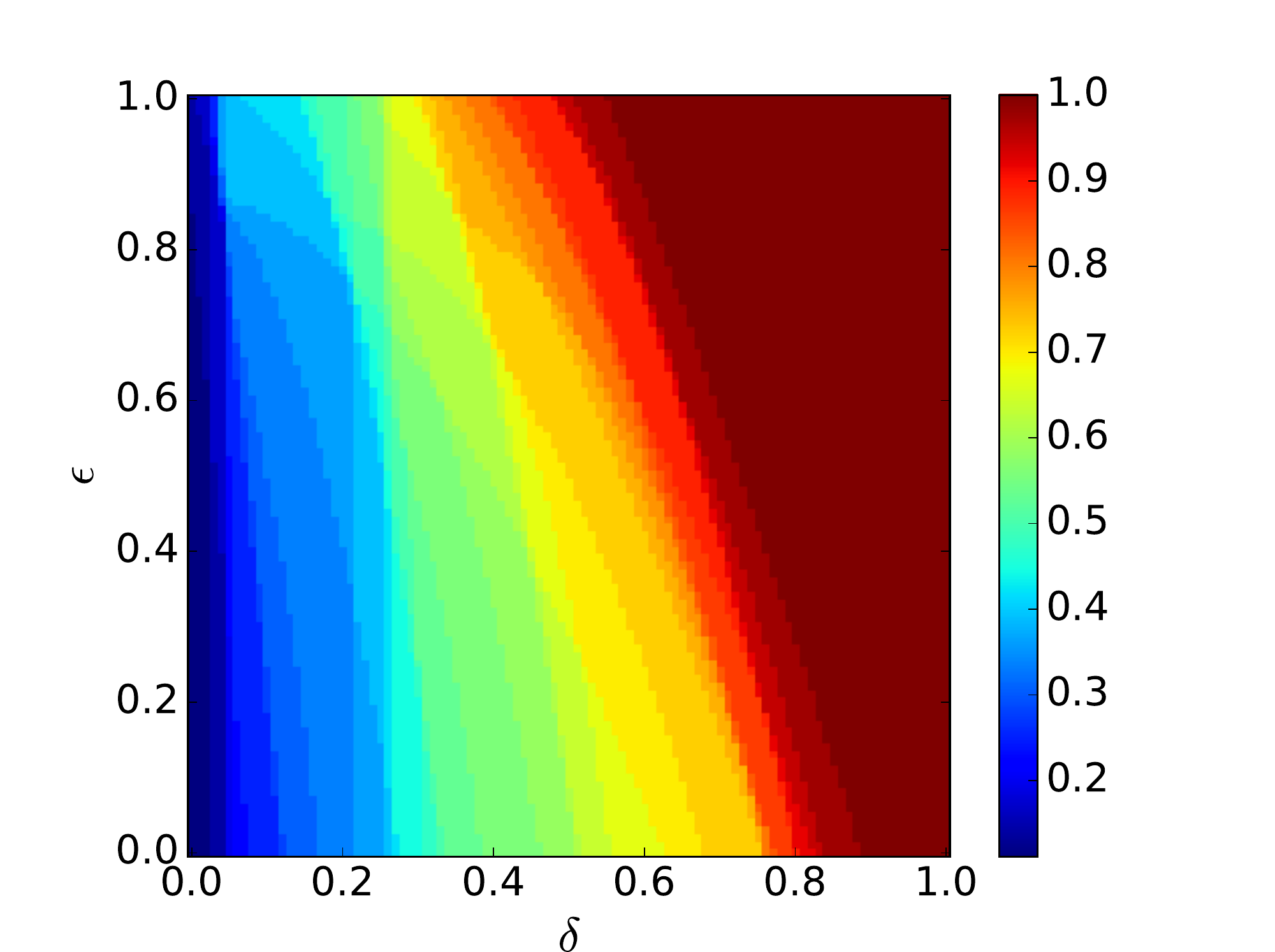}
		\includegraphics[width=.5\textwidth]{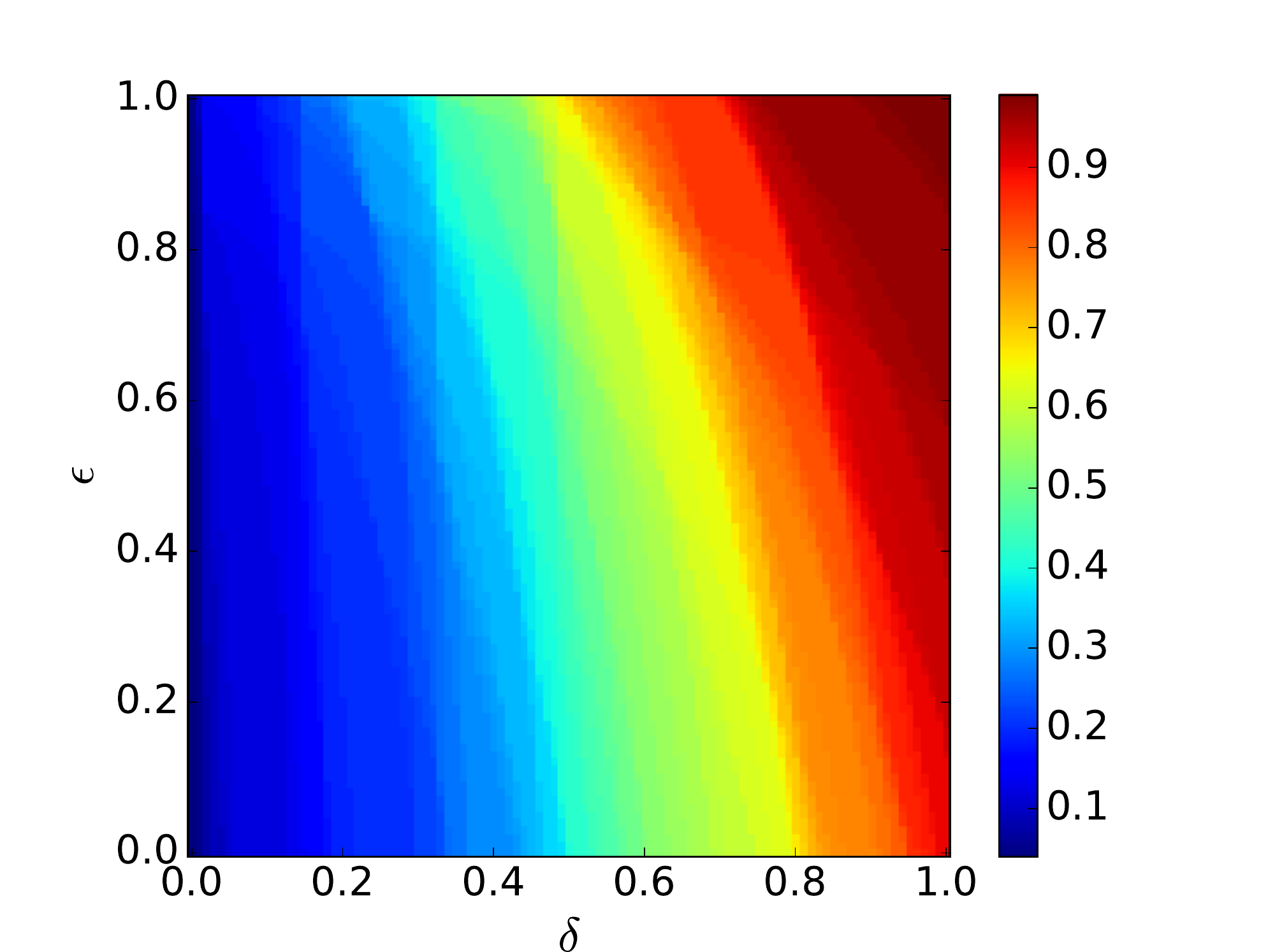}
		\caption{		\label{fig:heatmaps}
		Heat maps of the fraction of overflowing nodes for the family of networks
		of Example~1 over a range of $\delta$ and $\eps$. Left panel: $m = 3$ and $n = 36$. Right panel: $m = 5$ and $n = 100$.}
	\end{figure}

\subsection*{Example 1: Implementation of the overflow algorithm}

Starting with the computational example, we specify a family of networks $(\alpha, \mu, P_\delta, Q_\epsilon)_m$, which is parametrized by  $\delta \in [0, 1]$, $\epsilon \in [0, 1]$, and an integer $m \ge 2$. Given $m$, the number of nodes in the network is $n = 4m^2$. The matrix $P_{\delta}$ depends on $\delta$ and the matrix $Q_{\epsilon}$ depends on $\epsilon$ in a manner specified at the end of this example. It follows directly from that specification that $P_\delta 1 < 1$ and $Q_\epsilon 1 < 1$, so Condition~\ref{as:overflow} holds and hence Theorem~\ref{th:ourAlgorithmWorks} ensures that there is a unique solution to \eqref{eq:ourTrafficEquation} for every network $(\alpha, \mu, P_\delta, Q_\epsilon)_m$.

Our idea is to parameterize this family of networks in such a way that the number of overflowing nodes can significantly vary for different model parameters. For a given network specified by $(m, \epsilon, \delta)$ and a given solution $\lambda^\dagger$ to \eqref{eq:ourTrafficEquation}, there is a set $U^\dagger := \lbrace i \, : \, \lambda_i^\dagger \ge \mu_i \rbrace$ of overflowing nodes. Its magnitude $|U^\dagger|$ can range from $0$ to $n$ and we denote the proportion of overflowing nodes by
\[
{\cal F}_{m,\epsilon,\delta} := \frac{\big| U^\dagger \big|}{n}.
\]

Figure~\ref{fig:heatmaps} summarizes results from a computational experiment for $m=3$ ($n=36$) and $m=5$ ($n=100$). In the experiment, we vary $\epsilon$ and $\delta$ over a fine grid, and for each point we execute Algorithm~\ref{alg:overflowAlgorithm}, yielding $\lambda^\dagger$ and $U^\dagger$, for plotting a heat map of ${\cal F}_{m,\epsilon,\delta}$. 
%
As can be seen from Figure~\ref{fig:heatmaps}, as we vary $\epsilon$ and $\delta$, the proportion of overflowing nodes ${\cal F}_{m,\epsilon,\delta}$ varies between $0$ and $1$ and in general ${\cal F}_{m,\epsilon,\delta}$  is non-decreasing in both $\epsilon$ and $\delta$.

As we used a grid of step size $0.01$ for both $\epsilon$ and $\delta$ in the computational experiment, each of the plots in the figure represents computational results from slightly more than $10,000$ executions of Algorithm~\ref{alg:overflowAlgorithm}. These executions took less than an hour on a standard 2019 desktop computer with a Python implementation of the algorithm (available at \cite{GitHub:overflow-algorithm}). Our implementation is based directly on the pseudo-code of Algorithm~\ref{alg:overflowAlgorithm} and does not employ any optimizations. For example, each iteration of the inner loop solves the linear equations \eqref{eq:innerLoopLinear}, without trying to make use of solutions from previous iterations or efficient matrix factorizations.

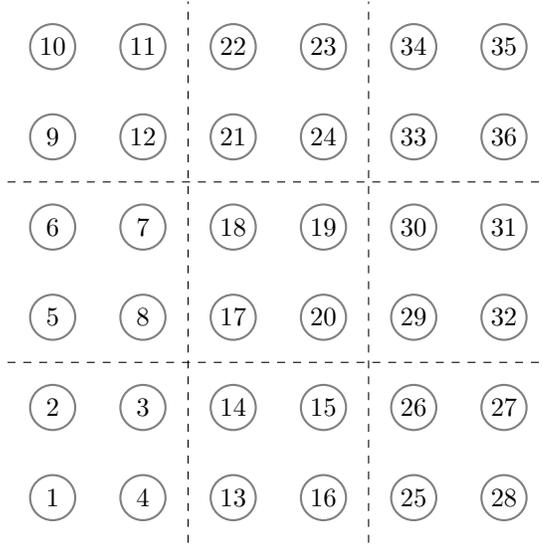
\begin{figure}
	\centering
\usetikzlibrary{arrows}
\begin{tikzpicture}[xscale=1.2,yscale=1.2]
	\tikzstyle{place} = [circle,draw=black!50,thick, inner sep=0pt,minimum size=6mm];
\tikzstyle{noline} = [circle,
inner sep=0pt,minimum size=7mm]

\draw [dashed] (0.5,2.5) -- (6.5,2.5);
\draw [dashed] (0.5,4.5) -- (6.5,4.5);

\draw [dashed] (2.5,0.5) -- (2.5,6.5);
\draw [dashed] (4.5,0.5) -- (4.5,6.5);

\node [place] (v1) at (1,1){1};
\node [place] (v2) at (1,2){2};
\node [place] (v3) at (2,2){3};
\node [place] (v4) at (2,1){4};

\node [place] (v5) at (1,3){5};
\node [place] (v6) at (1,4){6};
\node [place] (v7) at (2,4){7};
\node [place] (v8) at (2,3){8};

\node [place] (v9) at (1,5){9};
\node [place] (v10) at (1,6){10};
\node [place] (v11) at (2,6){11};
\node [place] (v12) at (2,5){12};


\node [place] (v13) at (3,1){13};
\node [place] (v14) at (3,2){14};
\node [place] (v15) at (4,2){15};
\node [place] (v16) at (4,1){16};

\node [place] (v17) at (3,3){17};
\node [place] (v18) at (3,4){18};
\node [place] (v19) at (4,4){19};
\node [place] (v20) at (4,3){20};

\node [place] (v21) at (3,5){21};
\node [place] (v22) at (3,6){22};
\node [place] (v23) at (4,6){23};
\node [place] (v24) at (4,5){24};


\node [place] (v25) at (5,1){25};
\node [place] (v26) at (5,2){26};
\node [place] (v27) at (6,2){27};
\node [place] (v28) at (6,1){28};

\node [place] (v29) at (5,3){29};
\node [place] (v30) at (5,4){30};
\node [place] (v31) at (6,4){31};
\node [place] (v32) at (6,3){32};

\node [place] (v33) at (5,5){33};
\node [place] (v34) at (5,6){34};
\node [place] (v35) at (6,6){35};
\node [place] (v36) at (6,5){36};
\end{tikzpicture}
		\caption{\label{fig:cellNet1} Layout of nodes in an $(\alpha, \mu, P_\delta, Q_\epsilon)_m$ network with $m=3$. There are $9$ cells and $4$ nodes per cell.}
	\end{figure}

To illustrate the effectiveness of our algorithm, let us consider an alternative simple brute-force method. Suppose that we try to solve \eqref{eq:ourTrafficEquation} for a single $(\epsilon, \delta)$ pair with $m=5$ ($n=100$). That is a single point in the right hand plot of Figure~\ref{fig:heatmaps}. The brute-force method considers every $S \subset N$ and finds the solution to the linear equation \eqref{eq:innerLoopLinear} with $S$ and $B = N \setminus S$. Then it checks whether that solution to \eqref{eq:innerLoopLinear} also satisfies \eqref{eq:ourTrafficEquation} and if it does, it stops and returns the solution to \eqref{eq:ourTrafficEquation}. Such a method only works well for problems with very small $n$. However, in our case there are $2^{100}$ subsets and this brute force method would effectively never succeed.

For reproducibility, we end this example by explicitly defining the family of networks $(\alpha, \mu, P_\delta, Q_\epsilon)_m$, consistent with its numerical implementation at \cite{GitHub:overflow-algorithm}. For any $m \ge 2$, the nodes are best represented in a rectangular grid made of $m \times m$ {\em cells} where each cell has $4$ nodes. We plot this for the case of $m=3$ ($n=36$) in Figure~\ref{fig:cellNet1}. Schematically cell boundaries are marked by dashed lines and nodes are marked by circles with numbering following the pattern as in the figure.

For these networks the input vector $\alpha = \begin{bmatrix} n^{2} & 0 & \dotsc & 0 \end{bmatrix}$ and the capacity vector $\mu = \begin{bmatrix} 1 & 1 & \dotsc & 1 \end{bmatrix}$. In other words, all nodes have unit capacity and input only arrives to the `south-west' node as appearing in Figure~\ref{fig:cellNet1} at a rate equal to the total capacity of all nodes.

Now consider the parameters $\delta$ and $\epsilon$ that yield different behaviors. The parameter $\delta$ regulates certain flows between nodes via the routing matrix $P_{\delta}$, while $\eps$ regulates certain overflows between nodes via the overflow matrix $Q_\epsilon$. These two $n \times n$ matrices are mostly sparse and only allow flows between a few neighboring nodes (where `neighboring' is defined according to the layout of Figure~\ref{fig:cellNet1}).


We specify the exact values of $P_\epsilon$ and $Q_\delta$ in Figure~\ref{fig:cellNetPQ}, where we consider the four nodes of an arbitrary cell and illustrate flows within the cell and to certain nodes in adjacent cells. The pattern of the network follows this structure throughout, with the exception that boundary effects should be ignored if there is no neighboring cell/node. Consider for example $Q_\epsilon$ and inspect both Figure~\ref{fig:cellNet1} and Figure~\ref{fig:cellNetPQ} (right). Then an (over)flow of $\epsilon$ exists between node $20$ and node $29$ (in the adjacent cell to the `east'). Similarly, there is an (over)flow of $\epsilon$ from $8$ to $17$. However, there no such (over)flow coming from the west into $5$, because its cell is on the western boundary. 

	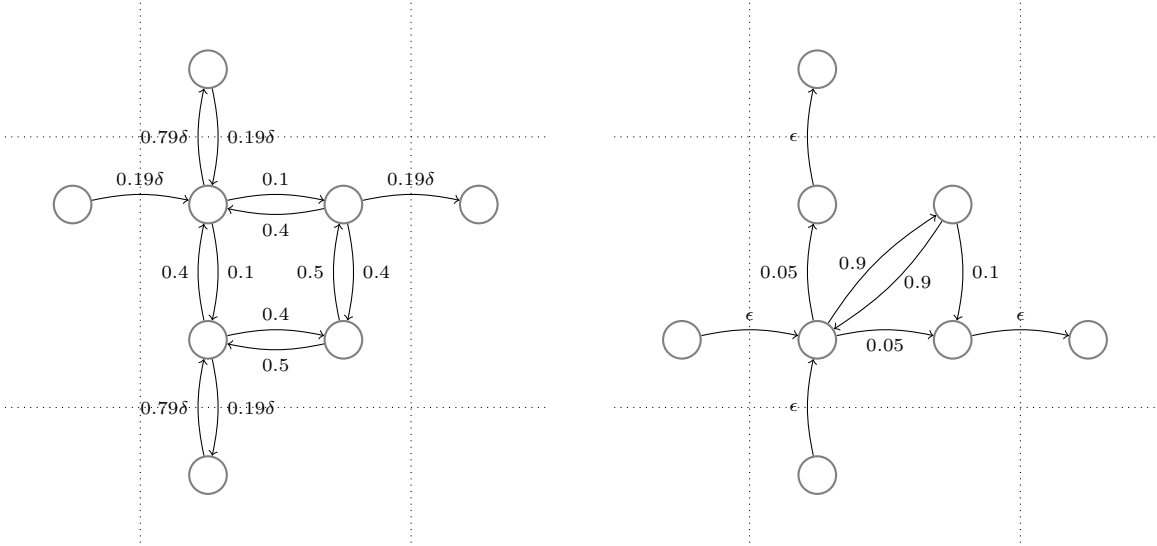
\begin{figure}
	\centering
\usetikzlibrary{arrows}
\begin{tikzpicture}[xscale=1.8,yscale=1.8]
	\tikzstyle{place} = [circle,draw=black!50,thick, inner sep=0pt,minimum size=5mm];
\tikzstyle{noline} = [circle,
inner sep=0pt,minimum size=7mm]

\draw [dotted] (1.5,2.5) -- (5.5,2.5);
\draw [dotted] (1.5,4.5) -- (5.5,4.5);
\draw [dotted] (2.5,1.5) -- (2.5,5.5);
\draw [dotted] (4.5,1.5) -- (4.5,5.5);


\node [place] (v7) at (2,4){};
\node [place] (v14) at (3,2){};
\node [place] (v17) at (3,3){};
\node [place] (v18) at (3,4){};
\node [place] (v19) at (4,4){};
\node [place] (v20) at (4,3){};
\node [place] (v21) at (3,5){};
\node [place] (v30) at (5,4){};

\draw[bend left = 12, ->] (v7) to node [auto] {\scriptsize $0.19\delta$} (v18);

\draw[bend left = 12, ->] (v18) to node [auto] {\scriptsize $0.79\delta$} (v21);
\draw[bend left = 12, ->] (v21) to node [auto] {\scriptsize $0.19\delta$} (v18);

\draw[bend left = 12, ->] (v18) to node [auto] {\scriptsize $0.1$} (v19);
\draw[bend left = 12, ->] (v19) to node [auto] {\scriptsize $0.4$} (v18);

\draw[bend left = 12, ->] (v17) to node [auto] {\scriptsize $0.4$} (v18);
\draw[bend left = 12, ->] (v18) to node [auto] {\scriptsize $0.1$} (v17);

\draw[bend left = 12, ->] (v17) to node [auto] {\scriptsize $0.4$} (v20);
\draw[bend left = 12, ->] (v20) to node [auto] {\scriptsize $0.5$} (v17);

\draw[bend left = 12, ->] (v20) to node [auto] {\scriptsize $0.5$} (v19);
\draw[bend left = 12, ->] (v19) to node [auto] {\scriptsize $0.4$} (v20);

\draw[bend left = 12, ->] (v19) to node [auto] {\scriptsize $0.19\delta$} (v30);

\draw[bend left = 12, ->] (v14) to node [auto] {\scriptsize $0.79\delta$} (v17);
\draw[bend left = 12, ->] (v17) to node [auto] {\scriptsize $0.19\delta$} (v14);

\newcommand{\ttm}{+4.5}

\draw [dotted] (1.5\ttm,2.5) -- (5.5\ttm,2.5);
\draw [dotted] (1.5\ttm,4.5) -- (5.5\ttm,4.5);
\draw [dotted] (2.5\ttm,1.5) -- (2.5\ttm,5.5);
\draw [dotted] (4.5\ttm,1.5) -- (4.5\ttm,5.5);

\node [place] (vb8) at (2\ttm,3){};
\node [place] (vb14) at (3\ttm,2){};
\node [place] (vb17) at (3\ttm,3){};
\node [place] (vb18) at (3\ttm,4){};
\node [place] (vb19) at (4\ttm,4){};
\node [place] (vb20) at (4\ttm,3){};
\node [place] (vb21) at (3\ttm,5){};
\node [place] (vb29) at (5\ttm,3){};

\draw[bend left = 12, ->] (vb14) to node [auto] {\scriptsize $\epsilon$} (vb17);
\draw[bend left = 12, ->] (vb17) to node [auto] {\scriptsize $0.05$} (vb18);
\draw[bend left = 12, ->] (vb18) to node [auto] {\scriptsize $\epsilon$} (vb21);

\draw[bend left = 12, ->] (vb8) to node [auto] {\scriptsize $\epsilon$} (vb17);
\draw[bend left = 12, ->] (vb17) to node [below] {\scriptsize $0.05$} (vb20);
\draw[bend left = 12, ->] (vb20) to node [auto] {\scriptsize $\epsilon$} (vb29);

\draw[bend left = 12, ->] (vb17) to node [left] {\scriptsize $0.9$} (vb19);
\draw[bend left = 12, ->] (vb19) to node [right] {\scriptsize $0.9$} (vb17);

\draw[bend left = 12, ->] (vb19) to node [right] {\scriptsize $0.1$} (vb20);

\end{tikzpicture}

		\caption{\label{fig:cellNetPQ} Structure of the matrices $P_\delta$ (left) and $Q_\epsilon$ (right) for a generic cell.}
	\end{figure}

\subsection*{Example 2: A Worst Case Scenario for Running Time}

In this example we create a family of networks parameterized by the number of nodes $n$. Each network has the property that it requires the maximum number of iterations in the overflow algorithm as specified in Theorem~\ref{th:ourAlgorithmWorks}, specifically in \eqref{eq:maxNumIterations}.

Fix $n$ and define a network of $n$ nodes as follows. The input vector $\alpha = \begin{bmatrix} 1 & 0 & \dotsc & 0 \end{bmatrix}$, while the capacity vector $\mu$ is specified by taking
\begin{align*}
	\mu_{i} =
	\begin{cases}
		1 + \frac{n - i}{n (n + 1)} & \text{ if } i < n,\\
		\frac{1}{2n} & \text{ if } i = n.
	\end{cases}
\end{align*}
We define the $n \times n$ routing matrix $P$ and the overflow matrix $Q$ via
\begin{align*}
	P =
	\begin{bmatrix}
		0 & 1 & 0 & & 0 & 0 \\
		0 & 0 & 1 & & 0 & 0 \\
		\vdots & & & \ddots & & \\
		0 & 0 & 0 & & 1 & 0 \\
		0 & 0 & 0 & & 0 & 1 \\
		0 & 0 & 0 & & 0 & 0 \\
	\end{bmatrix}
	\qquad \text{ and } \qquad
	Q =
	\begin{bmatrix}
		0 & 0 & & 0 & 0 & 0 \\
		q_{n} & 0 & & 0 & 0 & 0 \\
		0 & q_{n} & & & & \\
		\vdots & & \ddots & & & \\
		0 & 0 & & q_{n} & 0 & 0 \\
		0 & 0 & & 0 & q_{n} & 0 \\
	\end{bmatrix},
\end{align*}
where $q_{n} = 1 - 2^{-(n + 1)}$. By considering combinations of rows of $P$ and $Q$ as per \eqref{eq:overflowCondition},  it is easy to see that Condition~\ref{as:overflow} is satisfied.
It can be shown analytically that solving the overflow traffic equation~\eqref{eq:ourTrafficEquation} for this network $(\alpha,\mu, P, Q)$ requires the maximum number of iterations of Algorithm~\ref{alg:overflowAlgorithm}.

We also executed Algorithm~\ref{alg:overflowAlgorithm} for this family of networks with $n = 1, 2, \dotsc, 100$ and on each execution we verified that the number of iterations of the inner loop was exactly the maximum number of iterations given in \eqref{eq:maxNumIterations} and that $\lambda^\dagger$ solved~\eqref{eq:ourTrafficEquation}. A script for this execution is part of the GitHub repository \cite{GitHub:overflow-algorithm}.

\subsection*{Example 3: From \textsc{FD} to \textsc{NI}}

In this example and the next, we explore why we need Condition~\ref{as:filledordrained} and Condition~\ref{as:overflow}. We first consider a very simple network without overflows that does not satisfy the \textsc{fd} condition from \cite{goodman1984nej} but does satisfy the \textsc{ni} condition from Condition~\ref{as:filledordrained}.
	Consider the following network $(\alpha, \mu, P)$ with 4 nodes. The exogenous arrival vector is given by $\alpha = \begin{bmatrix} 1 & 0 & 0 & 0 \end{bmatrix}$ and the vector of service capacities is given by $\mu = \begin{bmatrix} 1 & 1 & 1 & 1 \end{bmatrix}$, while the routing matrix is defined via
	\begin{align*}
		P = \begin{bmatrix} 0 & 1 & 0 & 0 \\ 1 & 0 & 0 & 0 \\ 0 & 1/2 & 0 & 1/2 \\ 0 & 0 & 1 & 0 \end{bmatrix}.
	\end{align*}
	This matrix has two communicating classes $C_{1} = \lbrace 1, 2 \rbrace$ and $C_{2} = \lbrace 3, 4 \rbrace$. Only $C_{1}$ can be filled, while neither class can be externally drained. The network is therefore not \textsc{fd}. However, it is \textsc{ni}, since $C_{2}$ can be internally drained.

\subsection*{Example 4: When Condition~\ref{as:overflow} Breaks}

This final example deals with an overflow network that does not satisfy Condition~\ref{as:overflow}. We demonstrate that in this case the overflow traffic equation~\eqref{eq:ourTrafficEquation} may admit no solution, a unique solution, or infinitely many solutions. This example shows in particular that the existence of a unique solution to \eqref{eq:ourTrafficEquation} depends crucially on the magnitude of the network parameters. This is in sharp contrast with the case without overflows, in which the existence of a unique solution to the traffic equation \eqref{eq:goodmanMasseyEq} is completely determined by the incidence vectors and matrices corresponding to the network parameters.

	We consider a network $(\alpha, \mu, P, Q)$ with three nodes, specified as follows. The exogenous arrival vector is $\alpha = \begin{bmatrix} \alpha_{1} & 0 & 0 \end{bmatrix}$, where $\alpha_{1} > 0$. The vector of service capacities is $\mu = \begin{bmatrix} 4/3 & 2/3 & 1 \end{bmatrix}$. We define the matrices $P$ and $Q$ via
	\begin{align*}
		P = \begin{bmatrix} 0 & 1/2 & 0 \\ 1/2 & 0 & 0 \\ 1/2 & 1/2 & 0 \end{bmatrix} \quad \text{ and } \quad Q = \begin{bmatrix} 0 & 1/2 & 1/2 \\ 1/2 & 0 & 1/2 \\ 0 & 0 & 0 \end{bmatrix}.
	\end{align*}
	Since all nodes can be externally drained, the network $(\alpha, \mu, P)$ satisfies Condition~\ref{as:filledordrained}. We also have $\specrad{P} < 1$ and $\specrad{Q} < 1$. However, Condition~\ref{as:overflow} is never satisfied and the existence and uniqueness of a solution to the overflow equation \eqref{eq:ourTrafficEquation} for $(\alpha, \mu, P, Q)$ crucially depends on the value of $\alpha_{1}$. The underlying problem here is that the matrix	
	\begin{align}
		\label{eq:specrad1MatrixExample}
		P_{\lbrace 3 \rbrace} + Q_{\lbrace 1, 2 \rbrace} =
		\begin{bmatrix}
			0 & 1/2 & 1/2 \\
			1/2 & 0 & 1/2 \\
			1/2 & 1/2 & 0
		\end{bmatrix}
	\end{align}
	has spectral radius $1$.
	
	We now present the three different cases based on the value of $\alpha_{1}$ and indicate how $\alpha_{1}$ affects the existence and uniqueness of a solution to \eqref{eq:ourTrafficEquation}. For brevity, we focus on the results only and do not include formal proofs.
	
	If $\alpha_{1} < 1$, then the unique solution $\lambda^{*}$ to the traffic equation \eqref{eq:goodmanMasseyEq} is
	\begin{align*}
		\lambda^{*} = \alpha (I - P)^{-1} = \begin{bmatrix} 4 \alpha_{1} / 3 & 2 \alpha_{1} / 3 & 0 \end{bmatrix} < \mu,
	\end{align*}
	so $\lambda^{*}$ also solves the overflow traffic equation~\eqref{eq:ourTrafficEquation}. Condition~\ref{as:overflow} is not satisfied in this case due to \eqref{eq:specrad1MatrixExample}. Nevertheless, it can be verified that $\lambda^{*}$ is also the unique solution to \eqref{eq:ourTrafficEquation}. One way to do this is to manually check all $2^3$ linear equations \eqref{eq:innerLoopLinear} related to the network.
	
		
	If $\alpha_{1} = 1$, then
	\begin{align*}
		\lambda^{*} = \alpha (I - P)^{-1} = \begin{bmatrix} 4/3 & 2/3 & 0 \end{bmatrix} \leq \begin{bmatrix} 4/3 & 2/3 & 1 \end{bmatrix}  = \mu
	\end{align*}
	is a solution to \eqref{eq:ourTrafficEquation}. As in the previous case, Condition~\ref{as:overflow} is not satisfied due to \eqref{eq:specrad1MatrixExample}. However, in the current case, there are infinitely many solutions to \eqref{eq:ourTrafficEquation}. Every vector of the form $\lambda = \lambda^{*} + \begin{bmatrix} \eps & \eps & \eps \end{bmatrix}$ with $\eps \in [0, 1]$ can be verified to solve \eqref{eq:ourTrafficEquation}.
	
	Finally, if $\alpha_{1} > 1$, then there does not exist a solution to \eqref{eq:ourTrafficEquation}. One way to verify this is to check all $2^3$ linear equations \eqref{eq:innerLoopLinear} related to the network.

\section{Concluding remarks}
\label{sec:conc}

Motivated by queueing networks with finite buffers and overflows, we introduced the overflow traffic equation~\eqref{eq:ourTrafficEquation} as a natural generalization to the well-studied traffic equation~\eqref{eq:goodmanMasseyEq}. We developed a computationally efficient algorithm for solving the more general overflow equation. Such a solution indicates which buffers are overflowing and which are not and may then be used as part of a bottleneck analysis for the network. Our analysis showed that a characterization of existence and uniqueness of a solution requires care. In the process, we also refined understanding of equation~\eqref{eq:goodmanMasseyEq} by fully characterizing existence and uniqueness of nonnegative solutions via our \textsc{ni} condition.

It is important to note that our more general equation \eqref{eq:ourTrafficEquation} does not directly support an exact stationary distribution of a stochastic model as \eqref{eq:goodmanMasseyEq} does. 
This differs from say, stable Jackson networks, in which the stationary distribution of buffer $i$ is geometrically distributed with parameter $\rho_i=\lambda_i/\mu_i$, where $\lambda_i$ is determined by \eqref{eq:basicJackson} or \eqref{eq:goodmanMasseyEq}.  In such a stable Jackson network, the output rate of a stable node is exactly $\lambda_i$.

As an illustration of the impact of finite buffers, assume momentarily that a node $i$ is a standard M/M/1/K queue, as would be the case for an exponential service node receiving only exogenous Poisson arrivals at rate $\lambda_i=\alpha_i>0$. In this case, it is well-known from elementary queueing theory that
\begin{align*}
	\lim_{t \to \infty} \frac{L_i(t)}{t} = \lambda_i \frac{\rho_i^K - \rho_i^{K+1}}{1-\rho_i^{K+1}}
\end{align*}
with probability $1$, where $L_i(t)$ is the number of overflowing jobs from node $i$ (jobs arriving to node $i$ and finding the buffer of size $K$ full). For $\rho_i=1$ the limit is $\lambda_i/(K+1)$. This overflow rate, taken as a function of $\lambda_i$ only behaves as $\max(\lambda_i-\mu_i,0)$ for large $K$, yet the two terms are significantly different when $\lambda_i \approx \mu_i$ and $K$ is finite and not large. Nevertheless, for large buffered systems with small, fast jobs, it is plausible to use a fluid approximation yielding \eqref{eq:ourTrafficEquation}. We leave such an asymptotic analysis of discrete-job, stochastic networks for future research and in the current paper we derive results for traffic equation \eqref{eq:ourTrafficEquation}.

Towards that end, dynamics of such networks may also be characterized by equations similar to \eqref{eq:ourTrafficEquation} that describe general flows of the network in a transient context, a task that we have not undertaken as part of the current study.
Of similar interest, stemming from our results, is the relationship of our overflow algorithm to the general study of the Linear Complementarity Problem (LCP). We believe that a generalization of our algorithm may be framed in the context of LCP for solving a specific subclass of LCP problems that has still not been characterized. Undertaking such research remains a task for the future.

A related, more specific question deals with refining Condition~\ref{as:overflow} to strengthen Theorem~\ref{th:ourAlgorithmWorks} and to fully characterize existence and uniqueness of a nonnegative solution to \eqref{eq:ourTrafficEquation}. As we have shown in Example~4, there are cases in which Condition~\ref{as:overflow} is not satisfied, yet there exists a unique nonnegative solution to \eqref{eq:ourTrafficEquation}. Finding a natural characterization, similar to our \textsc{ni} condition for \eqref{eq:goodmanMasseyEq}, remains an open problem. Further, we conjecture that our overflow algorithm finds the unique solution to \eqref{eq:ourTrafficEquation} whenever such a solution exists. 

\vspace{20pt}

\noindent {\bf Acknowledgments}:  HMJ and YN are supported by Australian Research Council (ARC) Discovery Project DP180101602.


\appendix
\section{Auxiliary results for non-overflow networks}
\label{sec:technicalresults}

In this section, we prove a series of auxiliary results for networks $(\alpha, \mu, P)$ without overflows and the associated traffic equation \eqref{eq:goodmanMasseyEq}.

The first result is an immediate consequence of Tarski's fixed point theorem. It shows that there exists a nonnegative solution to \eqref{eq:goodmanMasseyEq} for every non-overflow network $(\alpha, \mu, P)$, regardless of any filled-or-drained conditions on $P$. A version of it also appeared in \cite{chen1991dfn}.

\begin{lemma}
	\label{lem:ourBoyTarski}
	Every network $(\alpha, \mu, P)$ admits a nonnegative solution to its traffic equation \eqref{eq:goodmanMasseyEq}.
\end{lemma}
\begin{proof}
	Consider the function $f (x) = \alpha + (x \wedge \mu) P$, which maps the complete lattice $[0, c]^{n}$ into itself if $c$ equals the largest entry of $\alpha + \mu P$. The partial order on this lattice is the natural order via entry-wise comparison and $f$ and preserves this order since $P$ is a non-negative matrix. Tarski's fixed point theorem \cite{tarski1955lattice} then implies that $f$ has at least one fixed point in $[0, c]^{n}$. This fixed point necessarily satisfies \eqref{eq:goodmanMasseyEq}.
\end{proof}

We next prove an elementary but useful result, which is utilized in the proof of Theorem \ref{th:GMalgorithmworks}.

\begin{lemma}
	\label{lem:triangularblockspecrad}
	Let $D_{1} , \dotsc , D_{l}$ be pairwise disjoint sets of nodes and denote their union by $D$. If each set $D_{k}$ contains nodes of exactly one class and the matrix $P_{D_{k} D_{k}}$ satisfies $\sigma (P_{D_{k} D_{k}}) < 1$, then both $\sigma (P_{D D}) < 1$ and $\sigma (P_{D}) < 1$.
\end{lemma}
\begin{proof}
	Since the sets of nodes $D_{1} , \dotsc , D_{l}$ are disjoint and each set contains nodes of exactly one class, we know that $P_{DD}$ is a block matrix with the blocks given by the matrices $P_{D_{k} D_{k}}$. Then the matrix $P_{DD}^{t}$ is the block matrix with blocks $P_{D_{k} D_{k}}^{t}$ for $t \in \mathbb{N}$. We also know that $P_{D_{k} D_{k}}^{t}$ converges to the zero matrix as $t \to \infty$, because $\sigma (P_{D_{k} D_{k}}) < 1$. We conclude that $P_{DD}^{t}$ converges to the zero matrix as well, so $\sigma (P_{D D}) < 1$. To establish the last statement, we observe that the block structure gives rise to the equality $P_{D}^{t + 1} = P_{DD}^{t} P_{D}$.
	Convergence of $P_{D}^{t}$ to the zero matrix follows immediately from this observation, so $\sigma (P_{D}) < 1$.
\end{proof}

We proceed by investigating existence and uniqueness of a nonnegative solution to the traffic equation~\eqref{eq:goodmanMasseyEq} under the \textsc{ni} condition. In the upcoming proofs, we use the following terminology. We say that a class $C_{1}$ has one-step access to another class $C_{2}$ if $p_{ij} > 0$ for some $i \in C_{1}$ and $j \in C_{2}$. We say that a class $C$ can access another class $D$ if there exist classes $C_{1} , \dotsc , C_{l}$ such that $C = C_{0}$, $D = C_{l}$, and $C_{k-1}$ has one-step access to $C_{k}$ for $k = 1 , \dotsc , l$. We say that a class $C$ is of level $h \in \mathbb{Z}_{\geqslant 0}$ if it can be accessed by exactly $h$ other classes.

A first observation is that there always exists at least one class of level $0$. Indeed, for fixed $h \in \mathbb{Z}_{\geqslant 0}$, a class of level $h+1$ cannot be accessed by a class of level $g > h$. Consequently, a class of level $h + 1$ has to be accessed by at least one class of level $g \leq h$, so there must be at least one class of level $0$.

\begin{lemma}
	\label{lem:nonnegsolutionforstochmat}
	Suppose that the network $(\alpha, \mu, P)$ is \textsc{ni} and that $P$ is an irreducible matrix. If $\lambda$ solves the traffic equation \eqref{eq:goodmanMasseyEq}, then $\lambda$ is nonnegative.
\end{lemma}
\begin{proof}
	Assume that $\lambda = \alpha + (\lambda \wedge \mu) P$ and take any empty sum to be $0$. Denoting $D^{-} = \lbrace i \in N \, \vert \, \lambda_{i} < 0 \rbrace$ and $D^{+} = N \setminus D^{-}$, we get
	\begin{align*}
		0
		&\leq \sum_{i \in D^{-}} \alpha_{i} \\
		&= \sum_{i \in D^{-}} \lambda_{i} - \sum_{i \in D^{-}} \sum_{j \in N} (\lambda_{j} \wedge \mu_{j}) p_{ji} \\
		&= \sum_{i \in D^{-}} \lambda_{i} - \sum_{j \in D^{-}} \lambda_{j} \sum_{i \in D^{-}} p_{ji} - \sum_{j \in D^{+}} (\lambda_{j} \wedge \mu_{j}) \sum_{i \in D^{-}} p_{ji} \\
		&\leq \sum_{i \in D^{-}} \lambda_{i} - \sum_{j \in D^{-}} \lambda_{j} \sum_{i \in D^{-}} p_{ji} \\
		&\leq 0.
	\end{align*}
	The penultimate inequality follows from the definition of $D^{+}$, while the last inequality follows from the definition of $D^{-}$.
	
	Consequently, we obtain the equality $\sum_{i \in D^{-}} \lambda_{i} - \sum_{j \in D^{-}} \lambda_{j} \sum_{i \in D^{-}} p_{ji} = 0$. This equality holds trivially if $D^{-}$ is empty. However, if $ D^{-}$ is nonempty, then the equality holds if and only if  $\sum_{i \in D^{-}} p_{ji} = 1$ for all $j \in D^{-}$. This last condition can only be true if $D^{-} = N$, because $P$ is irreducible. Thus, if $\lambda$ is not nonnegative, then it is negative and $P$ must be stochastic.
	
	Now assume that $\lambda$ is negative and that $P$ is stochastic. Then the network forms a class that cannot be drained, so by the \textsc{ni} condition it can be filled and thus $\sum_{i \in N} \alpha_{i} > 0$. In this case, we get
	\begin{align*}
		0 < \sum_{i \in N} \alpha_{i} = \sum_{i \in N} \lambda_{i} - \sum_{i \in N} \sum_{j \in N} \lambda_{i} p_{ij} = 0,
	\end{align*}
	which is a contradiction. We conclude that $\lambda$ must be nonnegative.
\end{proof}

%

\begin{lemma}
	\label{lem:fdnonnegsolposfil}
	Suppose that the network $(\alpha, \mu, P)$ is \textsc{ni}. If $\lambda$ solves the traffic equation \eqref{eq:goodmanMasseyEq}, then $\lambda$ is nonnegative. If $C$ is a class that can be filled, then $[\lambda]_{C} > 0$. If $C$ is a class that cannot be filled, then $[\lambda]_{C} = 0$.
\end{lemma}
\begin{proof}
	We first assume that $P$ is irreducible, so there is exactly one class. From Lemma \ref{lem:nonnegsolutionforstochmat} we know that $\lambda$ is nonnegative. Consider the case in which none of the nodes can be filled. Then $\alpha = 0$ and the network can be drained, so $P$ is not stochastic. Moreover, we get
	\begin{align*}
	\lambda = (\lambda \wedge \mu) P \leq \lambda P \leq \lambda P^{k}
	\end{align*}
	for $k \in \mathbb{N}$. Because $P$ is strictly substochastic and irreducible, the matrix $P^{k}$ converges to zero and thus $\lambda = 0$.
	
	Consider the case in which at least one node can be filled, so $\alpha_{i} > 0$ for at least one node $i$. An iteration of the traffic equation gives us
	\begin{align*}
	\lambda = \alpha + ( (\alpha + (\lambda \wedge \mu) P) \wedge \mu) P.
	\end{align*}
	Consequently, $\alpha_{i} > 0$ implies that $\lambda_{i} > 0$. Additionally, $\lambda_{i} > 0$ and $p_{ij} > 0$ implies that $\lambda_{j} > 0$. Because $P$ is irreducible, we conclude that $\lambda > 0$ in this case.
	
	We now drop the assumption that $P$ is irreducible and use an induction argument to prove the lemma. As a first step, consider a class $C$ of level $0$. Then
	\begin{align*}
	[\lambda]_{C} = [\alpha]_{C} + ([\lambda]_{C} \wedge [\mu]_{C}) [P]_{CC}
	\end{align*}
	and the network $([\alpha]_{C}, [\mu]_{C}, [P]_{CC})$ satisfies the \textsc{ni} condition. Indeed, if $C$ can be drained, then the \textsc{ni} condition is trivially satisfied. If $C$ cannot be drained, it can be filled and the \textsc{ni} condition is also satisfied. Because $C$ is a communicating class, we also know that $[P]_{CC}$ is irreducible. The previous arguments show that $[\lambda]_{C}$ is nonnegative in this case. Additionally, $[\lambda]_{C} > 0$ if $C$ can be filled and $[\lambda]_{C} = 0$ if $C$ cannot be filled.
	
	For the induction step, we assume that the following holds. For fixed $h \in \mathbb{N}$, every class $D$ of level $g < h$ satisfies $[\lambda]_{D} > 0$ if $D$ can be filled and $[\lambda]_{D} = 0$ if $D$ cannot be filled. Now consider a class $C$ of level $h$. This class can only be accessed by classes that have a level $g < h$. Let $D_{1} , \dotsc , D_{l}$ be the different classes that have one-step access to $C$. Then we can write $[\lambda]_{C}$ as
	\begin{align*}
	[\lambda]_{C} = \beta + ([\lambda]_{C} \wedge [\mu]_{C}) [P]_{CC},
	\end{align*}
	with
	\begin{align*}
	\beta = [\alpha]_{C} + \sum_{k=1}^{l} ([\lambda]_{D_{k}} \wedge [\mu]_{D_{k}}) [P]_{D_{k} C}.
	\end{align*}
	Suppose $C$ can be filled. Then there are two (possibly overlapping) cases. In the first case, $[\alpha]_{C} \not= 0$. In the second case, one of the classes $D_{1} , \dotsc , D_{l}$ can be filled and thus $[\lambda]_{D_{k}} > 0$ for some $k$. As a result, $\beta \not= 0$ in both cases, so the network $(\beta, [\mu]_{C},[P]_{CC})$ is \textsc{ni}.
	Suppose $C$ cannot be filled. Then $C$ can be drained and thus the network $(\beta, [\mu]_{C},[P]_{CC})$ is also \textsc{ni}. As before, we know that $[P]_{CC}$ is irreducible and we conclude that $[\lambda]_{C}$ is nonnegative. Additionally, $[\lambda]_{C} > 0$ if $C$ can be filled and $[\lambda]_{C} = 0$ if $C$ cannot be filled. This completes the induction step.
	Because we have verified the assumption of the induction step for the classes of level $0$, the statement of the lemma follows immediately.
\end{proof}

\begin{lemma}
	\label{lem:positivekappa}
	Suppose that the network $(\alpha, \mu, P)$ is \textsc{ni} and let $\gamma$ be a nonnegative vector satisfying
	\begin{align*}
		\gamma \geq \alpha + (\gamma \wedge \mu) P.
	\end{align*}
	If $C$ is a class that can be filled, then $[\gamma]_{C} > 0$. If $C$ can be filled but not drained, then $\gamma_{i} > \mu_{i}$ for some~$i \in C$. 
\end{lemma}
\begin{proof}
	Recall that $\abs{x} = \sum_{i=1}^{n} x_{i}$ for an $n$-dimensional nonnegative vector $x$. It is easy to see that $\abs{x} = \abs{x P}$ if $P$ is stochastic and that $\abs{x} > \abs{x P}$ if $P$ is not stochastic.
	
	The first statement follows from the proofs of Lemma \ref{lem:fdnonnegsolposfil}. Suppose that the class $C$ can be filled but not drained and let $A = N \setminus C$. Then
	\begin{align*}
		[\gamma]_{C} \geq \beta + ([\gamma]_{C} \wedge [\mu]_{C}) [P]_{CC}
	\end{align*}
	with
	\begin{align*}
		\beta = [\alpha]_{C} + ([\gamma]_{A} \wedge [\mu]_{A}) [P]_{AC}.
	\end{align*}
	Because $C$ can be filled, we know that $\beta \not= 0$. We also know that the matrix $[P]_{CC}$ is irreducible and stochastic, as $C$ cannot be drained. This implies that
	\begin{align*}
		\abs{[\gamma]_{C}}
		&\geq \abs{\beta} + \abs{([\gamma]_{C} \wedge [\mu]_{C}) [P]_{CC}} \\
		&> \abs{([\gamma]_{C} \wedge [\mu]_{C}) [P]_{CC}} \\
		&\geq \abs{[\gamma]_{C} \wedge [\mu]_{C}},
	\end{align*}
	so $\gamma_{i} > \mu_{i}$ for some $i \in C$.
	%
\end{proof}


The following lemma is used in the proof of Theorem~\ref{th:ourAlgorithmWorks} and it is also useful in its own-right because it states that  non-overflow networks are monotonous in the input rate.

\begin{lemma}
	\label{lem:goodmanmasseyordered}
	Suppose that the network $(\alpha, \mu, P)$ is \textsc{ni}. If $x = \alpha + (x \wedge \mu) P$ and $y = \alpha + \eps + (y \wedge \mu) P$ for some $\eps \geq 0$, then $y \geq x$.
\end{lemma}
\begin{proof}
	Let $C = \lbrace i \in N \, \vert \, y_{i} < x_{i} \rbrace$ and assume that $C$ is nonempty. Then
	\begin{align*}
		0 > \sum_{i \in C} (y_{i} - x_{i})
		&= \sum_{i \in C} \sum_{j \in N} (y_{j} \wedge \mu_{j} - x_{j} \wedge \mu_{j}) p_{ji} + \sum_{i \in C} \eps_{i} \\
		&\geq \sum_{i \in C} \sum_{j \in C} (y_{j} \wedge \mu_{j} - x_{j} \wedge \mu_{j}) p_{ji} + \sum_{i \in C} \eps_{i} \\
		&= \sum_{i \in C} \big( \sum_{j \in C} p_{ij} \big) (y_{i} \wedge \mu_{i} - x_{i} \wedge \mu_{i}) + \sum_{i \in C} \eps_{i},
	\end{align*}
	which implies that $\sum_{i \in C} \eps_{i} = 0$, $\sum_{j \in C} p_{ij} = 1$ for all $i \in C$, and $x_{i} \leq \mu_{i}$ for all $i \in C$. This means in particular that $C$ is a class that cannot be drained. Because $(\alpha, \mu, P)$ is \textsc{ni}, it follows from Lemma \ref{lem:positivekappa} that $x_{i} > \mu_{i}$ for some $i \in C$. We conclude from this contradiction that $C$ must be empty, so $y \geq x$.
\end{proof}

\bibliography{References}

\begin{thebibliography}{10}

\bibitem{asaduzzaman2010loss}
M.~Asaduzzaman, T.~J. Chaussalet, and N.~J. Robertson.
\newblock A loss network model with overflow for capacity planning of a
  neonatal unit.
\newblock {\em Annals of Operations Research}, 178(1):67--76, 2010.

\bibitem{balsamo2001aqn}
S.~Balsamo, V.~de~Nitto~Persone, and R.~O. Onvural.
\newblock {\em {Analysis of Queueing Networks with Blocking}}.
\newblock Kluwer Academic Publishers, 2001.

\bibitem{bramsonBook2008}
M.~Bramson.
\newblock {\em {Stability of Queueing Networks}}.
\newblock Springer, 2008.

\bibitem{chen1991dfn}
H.~Chen and A.~Mandelbaum.
\newblock {Discrete flow networks: bottleneck analysis and fluid
  approximations}.
\newblock {\em Math. Oper. Res}, 16(2):408--446, 1991.

\bibitem{bookChenYao2001}
H.~Chen and D.~D. Yao.
\newblock {\em {Fundamentals of Queueing Networks: Performance, Asymptotics,
  and Optimization}}.
\newblock Springer, 2001.

\bibitem{cottlelinear}
R.~W. Cottle, J.~S. Pang, and R.~E. Stone.
\newblock {\em The linear complementarity problem}.
\newblock Springer, 1992.

\bibitem{Dai108}
J.~G. Dai.
\newblock On positive {H}arris recurrence of multiclass queueing networks: A
  unified approach via fluid limit models.
\newblock {\em The Annals of Applied Probability}, 5(1):49--77, 1995.

\bibitem{dai1999heavy}
J.~G. Dai and W.~Dai.
\newblock A heavy traffic limit theorem for a class of open queueing networks
  with finite buffers.
\newblock {\em Queueing Systems}, 32(1-3):5--40, 1999.

\bibitem{GitHub:overflow-algorithm}
S.~Fleuren, H.~M. Jansen, E.~Lefeber, and Y.~Nazarathy.
\newblock Overflow algorithm.
\newblock https://github.com/hmjansen/overflow-algorithm.

\bibitem{goodman1984nej}
J.~B. Goodman and W.~Massey.
\newblock {Non-ergodic Jackson network.}
\newblock {\em Journal of Applied Probability}, 21(4):860--869, 1984.

\bibitem{GurvichPerry2012}
I.~Gurvich and O.~Perry.
\newblock Overflow networks: Approximations and implications to call center
  outsourcing.
\newblock {\em Operations Research}, 60(4):996--1009, 2012.

\bibitem{hordijk1987stochastic}
A.~Hordijk and A.~Ridder.
\newblock Stochastic inequalities for an overflow model.
\newblock {\em Journal of Applied Probability}, 24(3):696--708, 1987.

\bibitem{Jackson407}
J.~R. Jackson.
\newblock {Networks of waiting lines}.
\newblock {\em Operations Research}, 5(4):518--521, 1957.

\bibitem{Kaufmanetal1981}
L.~Kaufman, J.~B. Seery, and J.~A. Morrison.
\newblock Overflow models for dimension pbx feature packages.
\newblock {\em The Bell System Technical Journal}, 60(5):661--676, 1981.

\bibitem{Kelly0229}
F.~P. Kelly.
\newblock {Loss networks}.
\newblock {\em Ann. Appl. Probab}, 1(3):319--378, 1991.

\bibitem{KooleTalim2000}
G.~M. Koole and J.~Talim.
\newblock Exponential approximation of multi-skill call centers architecture.
\newblock {\em Proceedings of QNETs 2000}, pages 23/1--10, 2000.

\bibitem{Litvak2008}
N.~Litvak, M.~van Rijsbergen, R.~J. Boucherie, and M.~van Houdenhoven.
\newblock Managing the overflow of intensive care patients.
\newblock {\em European Journal of Operational Research}, 185(3):998--1010,
  2010.

\bibitem{murty1988linear}
K.~G. Murty.
\newblock {\em Linear complementarity, linear and nonlinear programming}.
\newblock Heldermann Berlin, 1988.

\bibitem{newell1971aqt}
G.~F. Newell.
\newblock {\em {Applications of Queueing Theory}}.
\newblock Chapman and Hall, 1982.

\bibitem{schweitzer1982bottleneck}
P.~J. Schweitzer.
\newblock Bottleneck determination in networks of queues.
\newblock In {\em Applied Probability-Computer Science: The Interface Volume
  1}, pages 471--485. Springer, 1982.

\bibitem{Sendfeld2008}
P.~Sendfeld.
\newblock Two queues with weighted one-way overflow.
\newblock {\em Methodology and Computing in Applied Probability},
  10(4):531--555, 2008.

\bibitem{sommer2017analysis}
J.~Sommer, J.~Berkhout, H.~Daduna, and B.~Heidergott.
\newblock {Analysis of Jackson networks with infinite supply and unreliable
  nodes}.
\newblock {\em Queueing Systems}, 87(1-2):181--207, 2017.

\bibitem{tarski1955lattice}
A.~Tarski.
\newblock A lattice-theoretical fixpoint theorem and its applications.
\newblock {\em Pacific Journal of Mathematics}, 5(2):285--309, 1955.

\bibitem{vanDijk1987}
N.~M. van Dijk.
\newblock Simple and insensitive bounds for a grading and an overflow model.
\newblock {\em Operations Research Letters}, 6(2):73--76, 1987.

\bibitem{vanDijk1989}
N.~M. van Dijk.
\newblock A proof of simple insensitive bounds for a pure overflow system.
\newblock {\em Journal of Applied Probability}, 26(1):113--120, 1989.

\bibitem{bookWolff1989}
R.~W. Wolff.
\newblock {\em {Stochastic Modeling and the Theory of Queues}}.
\newblock Prentice Hall, 1989.

\bibitem{zwart2000fluid}
A.~P. Zwart.
\newblock A fluid queue with a finite buffer and subexponential input.
\newblock {\em Advances in Applied Probability}, 32(1):221--243, 2000.

\end{thebibliography}
\end{document}